\documentclass[final]{siamltex}
\usepackage{amsfonts,amsmath,amssymb}
\usepackage{mathrsfs,mathtools}
\usepackage{enumerate}
\usepackage{xspace} 
\usepackage{hyperref}
\usepackage{esint}
\usepackage{graphicx}
\DeclareMathAlphabet{\mathpzc}{OT1}{pzc}{m}{it}

\newcommand{\R}{\mathbb{R}}

\newcommand{\Y}{\mathpzc{Y}}

\newcommand{\N}{\mathpzc{N}}

\newcommand{\C}{\mathcal{C}}

\newcommand{\V}{\mathbb{V}}

\newcommand{\Xcal}{\mathcal{X}}
\newcommand{\Ycal}{\mathcal{Y}}

\newcommand{\T}{\mathscr{T}}

\newcommand{\Ss}{\mathscr{S}}

\newcommand{\HL}{ \mbox{ \raisebox{7.0pt} {\tiny$\circ$} \kern-10.7pt} {H_L^1} }
\newcommand{\Wp}{ \mbox{ \raisebox{7.7pt} {\scriptsize$\circ$} \kern-10.1pt} {W^{1,p}} }
\newcommand{\Wpp}{ \mbox{ \raisebox{7.7pt} {\scriptsize$\circ$} \kern-10.1pt} {W^{1,p'}} }
\newcommand{\Sz}{ \mbox{ \raisebox{7.5pt} {\scriptsize$\circ$} \kern-10.1pt} {\Ss} }
\newcommand{\HLnew}{ \mbox{ \raisebox{7pt} {\scriptsize$\circ$} \kern-10.1pt}{H}^1_L }
\newcommand{\HLn}{{\mbox{\,\raisebox{5.1pt} {\tiny$\circ$} \kern-9.3pt}{H}^1_L  }}
\newcommand{\HLs}{{\mbox{\raisebox{8.7pt} {\scriptsize$\circ$} \kern-10.1pt}{H}^1_L  }}

\newcommand{\Tr}{\mathbb{T}}
\newcommand{\U}{\mathbb{U}}

\newcommand{\usf}{\mathsf{u}}
\newcommand{\asf}{\mathsf{a}}
\newcommand{\bsf}{\mathsf{b}}

\newcommand{\zsf}{\mathsf{z}}

\newcommand{\rsf}{\mathsf{r}}
\newcommand{\fsf}{\mathsf{f}}
\newcommand{\psf}{\mathsf{p}}
\newcommand{\gsf}{\mathsf{g}}

\newcommand{\orsf}{\bar{\mathsf{r}}}
\newcommand{\opsf}{\bar{\mathsf{p}}}
\newcommand{\ousf}{\bar{\mathsf{u}}}
\newcommand{\ozsf}{\bar{\mathsf{z}}}
\newcommand{\Zad}{\mathsf{Z}_{\textrm{ad}}}
\newcommand{\wsf}{\mathsf{w}}
\newcommand{\vsf}{\mathsf{v}}
\newcommand{\ue}{\mathscr{U}}
\newcommand{\oue}{\bar{\mathscr{U}}}
\newcommand{\pe}{\mathscr{P}}
\newcommand{\ope}{\bar{\mathscr{P}}}

\DeclareMathOperator*{\tr}{tr_\Omega}

\newcommand{\boxednumber}[1]{\expandafter\readdigit\the\numexpr#1\relax\relax}
\newcommand{\K}{\mathcal{K}}

\newcommand{\Hsd}{\mathbb{H}^{-s}(\Omega)}

\newcommand{\calLs}{\mathcal{L}^s}

\newcommand{\DIV}{\textrm{div}}
\newcommand{\diff}{\, \mbox{\rm d}}
\newcommand{\ie}{i.e.,\@\xspace}

\newcommand{\Hs}{\mathbb{H}^s(\Omega)}
\newcommand{\Ws}{\mathbb{H}^{1-s}(\Omega)}

\newcommand{\GL}{{\textup{\textsf{GL}}}}

\newcommand{\Nin}{\,{\mbox{\,\raisebox{6.0pt} {\tiny$\circ$} \kern-10.9pt}\N }}

\newcommand{\calT}{{\mathcal{T}}}

\newtheorem{remark}[theorem]{Remark}
\numberwithin{equation}{section}

\newcommand{\calL}{{\mathcal L}}
\DeclareMathOperator*{\esssup}{esssup}

\title{A fractional space-time optimal control problem: analysis and discretization\thanks{EO has been supported in part by NSF grants DMS-1109325 and DMS-1411808. AJS has been supported in part by NSF grant DMS-1418784.}}

\author{Harbir Antil\thanks{Department of Mathematical Sciences, George Mason University, Fairfax, VA 22030, USA. \texttt{hantil@gmu.edu}},
\and
Enrique Ot\'arola\thanks{Department of Mathematics, University of Maryland, College Park, MD 20742, USA and Department of Mathematical Sciences, George Mason University, Fairfax, VA 22030, USA. \texttt{kike@math.umd.edu}}
\and
Abner J.~Salgado\thanks{Department of Mathematics, University of Tennessee, Knoxville, TN 37996, USA. \texttt{asalgad1@utk.edu}}}

\date{Draft version of \today.}

\begin{document}

\maketitle
\begin{abstract}
We study a linear-quadratic optimal control problem involving a parabolic equation with fractional diffusion and Caputo fractional time derivative of orders $s \in (0,1)$ and $\gamma \in (0,1]$, respectively. The spatial fractional diffusion is realized as the Dirichlet-to-Neumann map for a nonuniformly elliptic operator. Thus, we consider an equivalent formulation with a quasi-stationary elliptic problem with a dynamic boundary condition as state equation. The rapid decay of the solution to this problem suggests a truncation that is suitable for numerical approximation. We consider a fully-discrete scheme: piecewise constant functions for the control and, for the state, first-degree tensor product finite elements in space and a finite difference discretization in time. We show convergence of this scheme and, for $s \in (0,1)$ and $\gamma = 1$, we derive a priori error estimates.
\end{abstract}

%
\section{Introduction}
\label{sec:introduccion}
We are interested in the design and analysis of efficient solution techniques for a linear-quadratic optimal control problem involving an initial boundary value problem for a space-time fractional parabolic equation. Let $\Omega$ be an open and bounded domain in $\R^n$ ($n\ge1$), with boundary $\partial\Omega$.
Given $s \in (0,1)$, $\gamma \in (0,1]$, and a desired state $\usf_d: \Omega \times (0,T) \rightarrow \mathbb{R}$,
we define
\[
 J(\usf,\zsf)= \frac{1}{2}\int_{0}^T \left( \| \usf - \usf_{d} \|^2_{L^2(\Omega)} + \mu \| \zsf\|^2_{L^2(\Omega)} \right) \diff t, 
\]
where $\mu > 0$ is the so-called regularization parameter.
Let $\fsf:\Omega \times (0,T) \rightarrow \mathbb{R}$ and $\usf_0:\Omega \to \mathbb{R}$ be fixed functions. 
We will call them the right hand side and initial datum, respectively. We shall be concerned with the following optimal control problem: Find 
\begin{equation}
\label{Jintro}
  \min J(\usf,\zsf),
\end{equation}
subject to the \emph{space-time fractional state equation}
\begin{equation}
\label{fractional_heat}
\partial^{\gamma}_t \usf + \mathcal{L}^s \usf = \fsf + \zsf  \text{ in } \Omega \times (0,T), \qquad \usf(0) = \usf_0 \text{ in } \Omega,
\end{equation}
and the \emph{control constraints}
\begin{equation}
\label{cc}
\asf(x',t) \leq \zsf(x',t) \leq \bsf(x',t) \quad\textrm{a.e.~~} (x',t) \in Q:= \Omega \times (0,T).
\end{equation}
The functions $\asf$ and $\bsf$ both belong to $L^2(Q)$ 
and satisfy the property $\asf(x',t) \leq \bsf	(x',t)$ for almost every $(x',t) \in Q$. The operator $\calLs$, with $s \in (0,1)$, is the fractional power of the second order elliptic operator 
\begin{equation}
\label{second_order}
 \mathcal{L} w = - \DIV_{x'} (A \nabla_{x'} w ) + c w \text{ in } \Omega, \qquad w = 0 \text{ on } \partial \Omega,
\end{equation}
where $0 \leq c \in L^\infty(\Omega)$ and $A \in C^{0,1}(\Omega,\GL(n,\R))$ is symmetric and positive definite.

The fractional derivative in time $\partial^{\gamma}_t$ for $\gamma \in (0,1)$ is understood as \textit{the left-sided Caputo fractional derivative of order $\gamma$} with respect to $t$, which is formally defined by
\begin{equation}
\label{caputo}
\partial^{\gamma}_t \usf(x',t) = \frac{1}{\Gamma(1-\gamma)} \int_{0}^t \frac{1}{(t-r)^{\gamma}} 
\frac{\partial \usf(x',r)}{\partial r}
\diff r,
\end{equation}
where $\Gamma$ is the Gamma function. For $\gamma = 1$, we consider the usual derivative $\partial_t$.

For convenience, we will refer to the optimal control problem \eqref{Jintro}--\eqref{cc} as the \emph{space-time fractional optimal control problem}; see section~\ref{sec:control} for its precise description and analysis. One of the main difficulties in the study of the state equation \eqref{fractional_heat} is the nonlocality of the fractional time derivative and the fractional space operator (see \cite{CS:07,CDDS:11,fractional_book,Samko,ST:10}). A possible approach to overcome the nonlocality in space is given by the result of Caffarelli and Silvestre in $\mathbb{R}^n$ \cite{CS:07} and its extensions to bounded domains \cite{CDDS:11,ST:10}: Fractional powers of the spatial operator $\mathcal{L}$ can be realized as an operator that maps a Dirichlet boundary condition to a Neumann condition via an extension problem on the semi-infinite cylinder $\C = \Omega \times (0,\infty)$. Therefore, we shall use the Caffarelli-Silvestre extension to rewrite the fractional space-time state equation \eqref{fractional_heat}
as a quasi-stationary elliptic problem with a dynamic boundary condition:
\begin{equation}
\label{heat_alpha_extension}
\begin{dcases}
-\DIV \left( y^{\alpha} \mathbf{A} \nabla \ue \right) + y^{\alpha} c\ue = 0 \textrm{ in } \C \times(0,T), & \ue = 0 \textrm{ on }\partial_L \C  \times (0,T),\\
\partial_t^{\gamma} \ue + \tfrac{1}{d_s}\partial_{\nu}^{\alpha} \ue  = \fsf + \zsf \textrm{ on } (\Omega \times \{ 0\}) \times (0,T),  
& \ue = \usf_0 \textrm{ on } \Omega \times \{ 0\}, ~t=0,
\end{dcases}
\end{equation}
where $\partial_L \C= \partial \Omega \times [0,\infty)$ is
the lateral boundary of $\C$, $\alpha =1-2s \in (-1,1)$, $d_s=2^\alpha \Gamma(1-s)/\Gamma(s)$ and the conormal exterior derivative of $\ue$ at $\Omega \times \{ 0 \}$ is
\begin{equation}
\label{def:lf}
\partial_{\nu}^{\alpha} \ue = -\lim_{y \rightarrow 0^+} y^\alpha \ue_y,
\end{equation}
where the limit must be understood in the distributional sense \cite{CS:07,CDDS:11,ST:10}. Finally, $\mathbf{A}(x',y) =  \textrm{diag} \{A(x'),1\}  \in C^{0,1}(\C,\GL(n+1,\R))$. We will call $y$ the \emph{extended variable} and the dimension $n+1$ in $\R_+^{n+1}$ the \emph{extended dimension} of problem \eqref{heat_alpha_extension}.
As noted in \cite{CS:07,CDDS:11,ST:10}, $\mathcal{L}^s$ and the Dirichlet-to-Neumann operator of \eqref{heat_alpha_extension} are related by
\[
 d_s \mathcal{L}^s \usf = \partial_{\nu}^{\alpha} \ue \quad \text{in } (\Omega \times \{ 0\}) \times (0,T).
\]
We briefly elaborate on these ideas in \S\ref{sub:stateequation}. A rigorous analysis is provided in \cite{NOS3,NOS}.

The study of solution techniques for elliptic and parabolic problems involving fractional derivatives is a relatively new but rapidly growing area of research. 
We refer the reader to \cite{NOS3,NOS,Otarola} for an overview of the state of the art. Numerical strategies for solving a discrete optimal control problem
with PDE constraints have been widely studied in the literature;  see \cite{HPUU:09,HT:10,IK:08,MV:08} for an extensive list of references. Mainly, these references are concerned with control problems governed by elliptic and parabolic PDEs, both linear and semilinear. The common feature here is that, in contrast to \eqref{Jintro}--\eqref{cc}, the state equation is local.

The numerical analysis of optimal control problems involving evolution equations with fractional diffusion and fractional time derivative is still at its infancy. 
To the best of our knowledge, the first work that provides a comprehensive treatment of an optimal control problem involving fractional elliptic operators in space is \cite{AO}. Concerning fractional derivatives in time, the first work that attempts to study an optimization problem constrained by a fractional order ODE is \cite{Agrawal:04} where, through completely formal calculations, the author derives optimality conditions and a numerical scheme. However, no justification is provided for either the optimality conditions nor the numerical scheme. Later, similar optimization problems have been discretized via a finite element method \cite{Agrawal:08}, a modified Gr\"unwald-Letnikov approach \cite{BDA:09,Defterli:10} and a rational approximation approach \cite{TC:10}. However, fundamental mathematical results such as stability and convergence of the proposed numerical schemes are missing in these works. Recently, convergence of spectral based techniques has been explored in 
\cite{LDY:11,LYD:13} for an optimization problem restricted to fractional order ODEs. Optimal control problems for one dimensional evolution equations with only fractional time derivatives  have been recently studied in \cite{YeXu:13,YeXu:14}. In these references, the authors derive rigorously first order necessary optimality conditions,
propose numerical schemes based on spectral methods and obtain a priori error estimates. These error estimates, however, are derived using regularity assumptions that are verified only in very restricted cases \cite{M:10,NOS3}.

We provide a comprehensive treatment of a linear-quadratic optimal control problem involving evolution equations with fractional diffusion and fractional time derivative: $s \in (0,1)$ and $\gamma \in (0,1]$. To the best of our knowledge this is the first work addressing such a problem from a mathematical point of view. We rigorously derive optimality conditions, present a numerical scheme and prove its convergence. In addition, for $s \in (0,1)$ and $\gamma = 1$, we derive a priori error estimates. We overcome the nonlocality of $\mathcal{L}^s$ by using the results of Caffarelli and Silvestre \cite{CS:07}. We realize the state equation \eqref{fractional_heat} by \eqref{heat_alpha_extension} so that, our problem can be equivalently written as: Minimize $J$ subject to the \emph{extended state equation} \eqref{heat_alpha_extension} and the control constraints \eqref{cc}.

Inspired by \cite{AO,NOS3,NOS}, we propose a simple strategy to find the solution to the space-time fractional optimal control problem \eqref{Jintro}--\eqref{cc}: given $\fsf$ and $\usf_d$, we realize \eqref{fractional_heat} by \eqref{heat_alpha_extension} and apply standard techniques to solve this problem. We thus obtain an optimal control $\bar{\zsf}: \Omega \times (0,T) \rightarrow \R$ and an optimal state $\bar{\ue}: \C \times (0,T) \rightarrow \R$. Letting $\bar{\usf}: \Omega \times (0,T) \ni (x',t) \mapsto \bar{\ue}(x',0,t) \in \R$ we obtain $(\bar{\usf},\bar{\zsf})$ that solves \eqref{Jintro}--\eqref{cc}.

The outline of this paper is as follows. In section~\ref{sec:Prelim} we introduce notation, recall elements from fractional calculus, define fractional powers of elliptic operators via spectral theory and show the equivalence with the Caffarelli-Silvestre extension. This allows us to study \eqref{heat_alpha_extension} and provide some energy estimates. On the basis of this, in section~\ref{sec:control}, we study the \emph{space-time fractional optimal control problem}. We derive existence and uniqueness results together with first order sufficient and necessary optimality conditions. In \S\ref{sec:control_truncated}, we begin the numerical analysis of our problem. We introduce a truncation of the state equation and derive approximation properties of its solution. In section~\ref{sec:a_priori_state}, we recall the fully discrete scheme of \cite{NOS3} that approximates the solution to the state equation \eqref{fractional_heat}. For $s\in (0,1)$ and $\gamma = 1$, we derive a novel $L^2(Q)$-error estimate in \S\ref{sec:fully_scheme_1}. Subsection~\ref{sub:fd_control} is devoted to the design of a numerical scheme to approximate the control problem \eqref{Jintro}--\eqref{cc}, and in \S\ref{sub:apriori_control}, we derive a priori error estimates for $s \in (0,1)$ and $\gamma = 1$. The convergence of the scheme is analyzed in \S\ref{sub:convergence} for $s \in (0,1)$ and $\gamma \in (0,1]$. Finally, section~\ref{s:numerics} presents numerical experiments that illustrate the theory developed in \S\ref{sub:apriori_control}.

\section{Notation and preliminaries}
\label{sec:Prelim}
Let us set notation and recall some facts that will be useful later.
\subsection{Notation}
\label{sub:notation}

Throughout this work $\Omega$ is an open, bounded and connected subset of $\R^n$, $n\geq1$, with polyhedral boundary $\partial\Omega$. If $T >0$ is a fixed time, we set $Q = \Omega \times (0,T)$. We will follow the notation of \cite{NOS3,NOS} and define the semi-infinite cylinder with base $\Omega$ and its lateral boundary, respectively, by $\C = \Omega \times (0,\infty)$ and $\partial_L \C  = \partial \Omega \times [0,\infty)$. For $\Y>0$, we define the truncated cylinder
$
  \C_\Y = \Omega \times (0,\Y)
$
and $\partial_L\C_\Y$ accordingly. Since we will be dealing with objects defined on $\R^n$ and $\R^{n+1}$, it will be convenient to distinguish the extended $n+1$-dimension. If $x\in \R^{n+1}$, we write
$\
  x =  (x',y),
$
with $x' \in \R^n$ and $y\in\R$.

If $\Xcal$ and $\Ycal$ are normed spaces, $\Xcal \hookrightarrow \Ycal$ means that $\Xcal$ is continuously embedded in $\Ycal$. We denote by $\Xcal'$ and $\|\cdot\|_{\Xcal}$ the dual and norm of $\Xcal$, respectively. The relation $a \lesssim b$ indicates that $a \leq Cb$, with a nonessential constant $C$ that might change at each occurrence. 

If $D\subset \R^{N}$ is open, $N \geq 1$, and $\phi: D \times (0,T) \to \R$, we will regard $\phi$ as a function of $t$ with values in a Banach space $\Xcal$, \ie
$
 \phi:(0,T) \ni t \mapsto  \phi(t) \equiv \phi(\cdot,t) \in \Xcal.
$
For $1 \leq p \leq \infty$, $L^p( 0,T; \Xcal )$ is the space of $\Xcal$-valued functions whose $\Xcal$-norm is in $L^p(0,T)$. This is a Banach space for the norm
\[
  \| \phi \|_{L^p( 0,T;\Xcal)} = \left( \int_0^T \| \phi(t) \|^p_\Xcal \diff t\right)^{\hspace{-0.1cm}\tfrac{1}{p}} 
  , \quad 1 \leq p < \infty, \quad
  \| \phi \|_{L^\infty( 0,T;\Xcal)} = \esssup_{t \in (0,T)} \| \phi(t) \|_\Xcal.
\]

\subsection{Fractional derivatives and integrals}
\label{sub:fractional_integral}

The left Caputo fractional derivative is defined in \eqref{caputo}. The \emph{right-sided Caputo fractional derivative} is \cite{fractional_book,Samko}:
\begin{equation}
\label{caputoR}
\partial^{\gamma}_{T-t} g(t) = - \frac{1}{\Gamma(1-\gamma)} \int_{t}^T \frac{g'(\xi)}{(\xi-t)^{\gamma}} 
 \diff \xi, \qquad \gamma \in (0,1).
\end{equation}
For $g \in L^1(0,T)$ and $\sigma > 0$, the \emph{left} and \emph{right Riemann-Liouville fractional integrals} of order $\sigma$ are, respectively, \cite[Definition 2.1, \S2]{Samko}
\begin{equation}
\label{fractional_integral}
(I_{t}^{\sigma} g)(t) = \frac{1}{\Gamma(\sigma)} \int_{0}^t \frac{g(\xi)}{(t-\xi)^{1-\sigma}} \diff \xi,
\qquad
(I_{T-t}^{\sigma} g)(t) = \frac{1}{\Gamma(\sigma)} \int_{t}^T \frac{g(\xi)}{(\xi-t)^{1-\sigma}} \diff \xi.
\end{equation}
\cite[\S2.2--2.3]{Samko} provides a motivation for these definitions inspired by the Abel equation.

\begin{proposition}[continuity of fractional integrals]
\label{pro:continuity}
For $\sigma > 0$ and $1 \leq p \leq \infty$, then 
$I_{t}^{\sigma}$ and $I_{T-t}^{\sigma}$ are continuous from $L^p(0,T)$ into itself and
\[
  \|I_t^{\sigma} g\|_{L^p(0,T)}\le \frac{T^\sigma}{\Gamma(\sigma+1)} \|g\|_{L^p(0,T)},
    \qquad 
  \|I_{T-t}^{\sigma} g\|_{L^p(0,T)}\le \frac{T^\sigma}{\Gamma(\sigma+1)} \|g\|_{L^p(0,T)}. 
\]
for all $g \in L^p(0,T)$. These maps also are continuous from $C([0,T])$ into itself.
\end{proposition}
\begin{proof}
For the proof of continuity in $L^p(0,T)$ see \cite[Theorem 2.6, \S2]{Samko}.
To obtain the continuity in $C([0,T])$ we use the continuity in $L^\infty(0,T)$, 
together with the fact that if $g \in C([0,T])$ then its fractional integrals are continuous as well. 
This can be easily shown by recalling that $g$ is also uniformly continuous.
\end{proof}

We also define the \emph{left} and \emph{right Riemann-Liouville fractional derivatives} of order $\gamma \in (0,1)$, respectively, by \cite[Definition 2.2, \S2.3]{Samko}
\[
  D_{t}^{\gamma} g(t) = \frac{1}{\Gamma(1-\gamma)} \frac{\diff}{\diff t}\int_{0}^t \frac{g(\xi)}{(t-\xi)^{\gamma}} \diff \xi,
  \ 
  D_{T-t}^{\gamma} g(t) = \frac{-1}{\Gamma(1-\gamma)} \frac{\diff}{\diff t} \int_{t}^T \frac{g(\xi)}{(\xi-t)^{\gamma}} \diff \xi.
\]
A relation between the Caputo and Riemann-Liouville derivatives is given below. 

\begin{lemma}[relation between fractional derivatives]
Let $\gamma \in (0,1)$ and $g \in W^1_1(0,T)$, then $D_t^{\gamma} g$ and $D_{T-t}^{\gamma}g$ exist almost everywhere on $[0,T]$. In addition,
$
D_t^{\gamma} g, D_{T-t}^{\gamma}g \in L^r(0,T)
$
for $1 \leq r < \tfrac{1}{\gamma}$, and
\begin{equation}
\label{relations}
 D_t^{\gamma} g(t) = \partial_t^{\gamma} g(t) + \frac{1}{\Gamma(1-\gamma)} \frac{g(0)}{t^{\gamma}},
 \quad 
 D_{T-t}^{\gamma} g(t) = \partial_{T-t}^{\gamma} g(t) + \frac{1}{\Gamma(1-\gamma)} \frac{g(T)}{(T-t)^{\gamma}}. 
\end{equation}
\end{lemma}
\begin{proof}
See \cite[Lemma 2.2, \S2.3]{Samko}.
\end{proof}

We now derive an integration by parts formula for Caputo derivatives that will be fundamental in our analysis. For $\gamma \in (0,1)$ we define
\[
  \mathbb{L}_\gamma = \{ f \in C([0,T]): \partial_t^{\gamma} f \in L^2(0,T)\},
  \quad
  \mathbb{R}_\gamma = \{ g \in C([0,T]): \partial_{T-t}^{\gamma} g \in L^2(0,T)\}.
\]

\begin{lemma}[fractional integration by parts formula]
\label{lem:frac_by_parts}
If $f \in \mathbb{L}_\gamma$ and $g \in \mathbb{R}_\gamma$, then the following fractional integration by parts holds:
\begin{equation}\label{fractional_by_parts}
 \int_{0}^{T} \partial_{t}^{\gamma} f(t) g(t) \diff t + f(0) ( I_{T-t}^{1-\gamma}g ) (0) =  
 \int_{0}^{T}  f(t) \partial_{T-t}^{\gamma} g(t) \diff t + g(T) ( I_{t}^{1-\gamma}f ) (T).
\end{equation}
\end{lemma}
\begin{proof}
If $f$ and $g$ are smooth, recall that \cite[Corollary 2, \S2.6]{Samko}:
\[
 \int_{0}^{T} D_{t}^{\gamma} f(t) g(t) \diff t  =  
 \int_{0}^{T}  f(t) D_{T-t}^{\gamma} g(t) \diff t;
\] 
\eqref{fractional_by_parts} now follows from \eqref{relations} and \eqref{fractional_integral}; the point values are well defined since $f \in C([0,T])$ implies $I_t^{\gamma}f,I_{T-t}^{\gamma}f \in C([0,T])$; see Proposition~\ref{pro:continuity}. Conclude by density.
\end{proof}

It is important to remark that there is another definition, not completely equivalent, of fractional derivatives: the so-called Gr\"unwald-Letnikov derivative \cite{fractional_book}. Among all possible definitions of fractional derivatives, we adopt the left-sided Caputo fractional derivative as $\partial_t^{\gamma}$ in problem \eqref{fractional_heat}: the Caputo approach leads to an initial condition of the form $\usf = \usf_0$ which is \emph{physically meaningful}. The Riemann-Liouville approach
leads to initial conditions containing the limit values of the Riemann-Liouville
fractional derivatives at $t=0$, something that does not have a clear physical meaning.

\subsection{Fractional powers of second order elliptic operators}
\label{sub:fractional_L}
Spectral theory for the operator $\mathcal{L}$ yields the existence of $\{(\lambda_k,\varphi_k)\}_{k \in \mathbb N} \subset \R^+ \times H^1_0(\Omega)$ such that
\begin{equation}
  \label{eigenvalue_problem_L}
    \mathcal{L} \varphi_k = \lambda_k \varphi_k  \text{ in } \Omega,
    \qquad
    \varphi_k = 0 \text{ on } \partial\Omega, \qquad k \in \mathbb N.
\end{equation}
$\{\varphi_k\}_{k \in \mathbb N}$ is an orthonormal basis of $L^2(\Omega)$. Fractional powers of  $\mathcal L$, are defined by
\begin{equation}
  \label{def:second_frac}
  \mathcal{L}^s w  := \sum_{k=1}^\infty \lambda_k^{s} w_k \varphi_k,
  \quad \forall w \in C_0^{\infty}(\Omega), \qquad s \in (0,1),
  \quad
  w_k = \int_{\Omega} w \varphi_k.
\end{equation}
By density, \eqref{def:second_frac} can be extended to $\Hs = [L^2(\Omega),H_0^1(\Omega)]_{s}$. If we denote by $\Hsd$ the dual of $\Hs$, then $\mathcal{L}^s : \Hs \to \Hsd$ is an isomorphism.

\subsection{Weighted Sobolev spaces}
\label{sub:weighted}

To study \eqref{heat_alpha_extension} we consider Sobolev spaces with the weight $|y|^{\alpha}$, $\alpha \in (-1,1)$. For $D \subset \R^{n+1}$ we define 
\[ 
  H^1(|y|^{\alpha},D) = \left\{ w \in L^2(|y|^{\alpha},D): | \nabla w | \in L^2(|y|^{\alpha},D) \right\},
\]
with norm
\begin{equation}
\label{wH1norm}
\| w \|_{H^1(|y|^{\alpha},D)} = \left( \| w \|^2_{L^2(|y|^{\alpha},D)} + \| \nabla w \|^2_{L^2(|y|^{\alpha},D)} \right)^{1/2}.
\end{equation}
Since $\alpha \in (-1,1)$, $|y|^\alpha$ belongs to the Muckenhoupt class $A_2(\R^{n+1})$; see \cite{GU,Turesson}. Then, $H^1(|y|^{\alpha},D)$ is Hilbert and $C^{\infty}(D) \cap H^1(|y|^{\alpha},D)$ is dense in $H^1(|y|^{\alpha},D)$ (cf.~\cite[Proposition 2.1.2, Corollary 2.1.6]{Turesson}, \cite{KO84} and \cite[Theorem~1]{GU}).

We also define the weighted Sobolev space
\begin{equation}
  \label{HL10}
  \HL(y^{\alpha},\C) = \left\{ w \in H^1(y^\alpha,\C): w = 0 \textrm{ on } \partial_L \C\right\}.
\end{equation}
As \cite[(2.21)]{NOS} shows, the following \emph{weighted Poincar\'e inequality} holds:
\begin{equation}
\label{Poincare_ineq}
\| w \|_{L^2(y^{\alpha},\C)} \lesssim \| \nabla w \|_{L^2(y^{\alpha},\C)}, \quad \forall w \in \HL(y^{\alpha},\C),
\end{equation}
thus $\| \nabla w \|_{L^2(y^{\alpha},\C)}$ is equivalent to \eqref{wH1norm} in $\HL(y^{\alpha},\C)$.
For $w \in H^1( y^{\alpha},\C)$, $\tr w$ denotes its trace onto $\Omega \times \{ 0 \}$. We recall (\cite[Prop. 2.5]{NOS})
\begin{equation}
\label{Trace_estimate}
\tr \HL(y^\alpha,\C) = \Hs,
\qquad
  \|\tr w\|_{\Hs} \lesssim
  \| w \|_{\HLn(y^\alpha,\C)}.
\end{equation}

\subsection{The state equation}
\label{sub:stateequation}

We follow \cite{NOS3} and define
\begin{align*}
  \mathbb{W} &:= \{ w \in L^{\infty}(0,T;L^2(\Omega)) \cap L^{2}(0,T;\Hs): \partial_t^{\gamma} 
    w \in L^2(0,T;\Hsd)\}, \\
  \mathbb{V} &:= \{ w \in L^{2}(0,T;\HL(y^{\alpha},\C)): \partial_t^{\gamma} \tr w  \in L^2(0,T;\Hsd)\}.  
\end{align*}
The Caffarelli-Silvestre extension result for problem \eqref{fractional_heat} reads \cite{CS:07,CDDS:11,ST:10,NOS3}: Given $\fsf, \zsf \in L^2(0,T;\Hsd)$, the function $\usf \in \mathbb{W}$ solves \eqref{fractional_heat} if and only if its harmonic extension $\ue \in \mathbb{V}$ solves the following version of \eqref{heat_alpha_extension}:
Find $\ue \in \mathbb{V}$ such that $\tr \ue(0) = \usf_0$ and for a.e.~$t \in (0,T)$
\begin{equation}
\label{heat_harmonic_extension_weak}
\langle \tr  \partial_t^{\gamma} \ue, \tr \phi \rangle + a(\ue,\phi) 
  = \langle \fsf + \zsf, \tr \phi \rangle \qquad \forall \phi \in \HL(y^{\alpha},\C),
\end{equation}
where $\langle \cdot, \cdot \rangle$ is the duality pairing between $\Hs$ and $\Hsd$ and
\begin{equation}
\label{a}
a(w,\phi) =   \frac{1}{d_s}\int_{\C} y^{\alpha}\left( \mathbf{A}(x') \nabla w \cdot \nabla \phi
               + c(x') w \phi \right) \diff x' \diff y.
\end{equation}
The regularity of $\mathbf{A}$ and $c$ implies that $a$ is bounded and coercive in $\HL(y^\alpha,\C)$. In what follows, we shall use repeatedly that $a(w,w)^{1/2}$ is a norm equivalent to $|\cdot|_{\HLn(y^\alpha,\C)}$.


Denote
\begin{align}
\label{Lambda}
\Lambda_{\gamma}^2(\vsf,\gsf) &:= I_t^{1-\gamma} \| \vsf \|_{L^2(\Omega)}^2(T) +  \| \gsf\|^2_{L^2(0,T;\Hsd)},
\\
\label{Sigma}
\Sigma^2 (\vsf,\gsf) &:= \| \vsf \|_{L^2(\Omega)}^2 +  \| \gsf \|^2_{L^2(0,T;\Hsd)},
\end{align}
where $I_t^{1-\gamma}$ is the left fractional integral of order $1-\gamma$ defined in \eqref{fractional_integral}.

\begin{theorem}[existence and uniqueness of $\usf$ and $\ue$]
\label{thm:exis_uniq}
Given $s \in (0,1)$, $\gamma \in (0,1]$, $\fsf, \zsf \in L^2(0,T;\Hsd)$ and $\usf_0 \in L^2(\Omega)$, problems \eqref{fractional_heat} and \eqref{heat_harmonic_extension_weak} have a unique solution. In addition, we have the following energy estimates for $\usf$, solution to \eqref{fractional_heat}:
\begin{align}
\label{stab_u}
I_t^{1-\gamma} \| \usf \|_{L^2(\Omega)}^2(T) + \|\usf \|^2_{L^2(0,T;\Hs)} 
& \lesssim \Lambda_{\gamma}^2(\usf_0,\fsf + \zsf),
\\
\label{energy_u}
\|\usf \|^2_{L^{^{\!\infty}\!}(0,T;L^2(\Omega))} + \|\usf \|^2_{L^2(0,T;\Hs)}
& \lesssim \Sigma^2 (\usf_0,\fsf + \zsf).
\end{align}
In addition, we have following energy estimates for $\ue$ solution to \eqref{heat_harmonic_extension_weak}:
\begin{align}
\label{stab_ue}
I_t^{1-\gamma} \| \tr \ue \|_{L^2(\Omega)}^2(T) + \| \ue \|^2_{L^2(0,T;\HLn(y^{\alpha},\C))} & \lesssim \Lambda_{\gamma}^2(\usf_0,\fsf + \zsf),
\\
\label{energy_ve}
\|\tr \ue \|^2_{L^{^{\!\infty}\!}(0,T;L^2(\Omega))} + \| \ue \|^2_{L^2(0,T;\HLn(y^{\alpha},\C))}
& \lesssim \Sigma^2 (\usf_0,\fsf + \zsf),
\end{align}
where the hidden constants do not depend on $\usf$, $\ue$ nor the problem data.
\end{theorem}
\begin{proof}
The well-posedness of \eqref{fractional_heat} and \eqref{heat_alpha_extension}, together with \eqref{stab_u} and \eqref{stab_ue} are presented in \cite[Theorem 2.6 and Corollary 3.8]{NOS3}. The estimates \eqref{energy_u} and \eqref{energy_ve} follow from the arguments developed in \cite{NOS3,Sakayama}.
\end{proof}

\begin{remark}[$\gamma=1$] \rm
\label{RE:militing_case}
Given $g \in L^p(0,T)$,  we have $I^\sigma g \rightarrow g$ in $L^p(0,T)$ as $\sigma \downarrow 0$ \cite[Theorem 2.6]{Samko}. Take the limit as $\gamma \uparrow 1$ in \eqref{stab_u} and \eqref{stab_ue}, to recover the well known energy estimates for parabolic equations with first order derivative in time.
\end{remark}

\begin{remark}[continuity in time] \label{rm:cont_t} \rm
An adaption of \cite[Theorems 2.1--2.2]{Sakayama} shows that, for every $\gamma \in (0,1]$ and $s \in (0,1)$, 
the solution $\tr \ue = \usf \in C([0,T];L^2(\Omega))$. This is not only necessary to make sense of the initial condition, but also to derive optimality conditions, as we will see in section~\ref{sec:control}.
\end{remark}

We conclude with an elementary extension of Lemma~\ref{lem:frac_by_parts}.

\begin{lemma}[fractional integration by parts]
\label{lem:ibp}
Let $\gamma \in (0,1]$. If $\vsf, \wsf \in \mathbb{W} \cap C([0,T];L^2(\Omega))$, then we have the following integration by parts formula:
\begin{multline*}
\int_0^T \langle \partial_t^{\gamma}\vsf(t), \wsf(t) \rangle - \langle \partial_{T-t}^{\gamma}\wsf(t), \vsf(t) \rangle  \diff t 
= ( \wsf(T) , (I_{t}^{1-\gamma}\vsf)(T) ) - ( \vsf(0) , (I_{T-t}^{1-\gamma}\wsf)(0)),
\end{multline*}
where $(\cdot, \cdot ) = (\cdot, \cdot )_{L^2(\Omega)}$ and $\langle \cdot, \cdot \rangle$ is the duality pairing between $\Hs$ and $\Hsd$.
\end{lemma}
\begin{proof}
When $\vsf$ and $\wsf$ are smooth we integrate \eqref{fractional_by_parts}. Conclude by density.
\end{proof}


\section{The fractional control problem}
\label{sec:control}

In this section, we analyze the \emph{space-time fractional optimal control problem}. We derive existence and uniqueness results together with first order necessary and sufficient optimality conditions.

For $J$ defined in \eqref{Jintro} the fractional control problem reads: Find
$
 \text{min } J(\usf,\zsf), 
$
subject to the state equation \eqref{fractional_heat} and the control constraints \eqref{cc}. The set of \emph{admissible controls} is defined by
\begin{equation}
 \label{ac}
 \Zad:= \left\{ \wsf \in L^2(Q):\ \asf(x',t) \leq \wsf(x',t) \leq \bsf(x',t), \ \textrm{a.e. }  (x',t) \in Q \right\},
\end{equation}
which is a nonempty, bounded, closed and convex subset of $L^2(Q)$. To study this problem, following \cite[\S3]{Tbook}, we introduce the control to state operator.

\begin{definition}[control to state operator]
\label{def:fractional_operator}
The map $\mathbf{S}: L^2(0,T;\Hsd) \ni \zsf \mapsto \usf(\zsf) \in \mathbb{W}$, where $\usf(\zsf)$ solves
\eqref{fractional_heat} is called the fractional control to state operator.
\end{definition}

$\mathbf{S}$ is an affine and, by the estimates of Theorem~\ref{thm:exis_uniq}, continuous operator. Moreover, since $\mathbb{W} \hookrightarrow L^2(Q) \hookrightarrow L^2(0,T;\Hsd)$, we may also consider the operator $\mathbf{S}$ as acting from $L^2(Q)$ into itself. For simplicity, we keep the notation $\mathbf{S}$. We now define the optimal fractional state-control pair.

\begin{definition}[optimal fractional state-control pair]
A state-control pair $(\ousf(\ozsf),\ozsf)$ $\in \mathbb{W} \times \Zad$ is called optimal for the 
problem \eqref{Jintro}--\eqref{cc}, if $\ousf(\ozsf) = \mathbf{S} \ozsf$ and
\[
 J(\ousf(\ozsf),\ozsf) \leq J(\usf(\zsf),\zsf),
\]
for all $(\usf(\zsf),\zsf) \in \mathbb{W} \times \Zad$ such that $\usf(\zsf)= \mathbf{S} \zsf$.
\label{def:fractional_pair}
\end{definition}

The existence and uniqueness of an optimal state-control pair is as follows.

\begin{theorem}[existence and uniqueness]
\label{TH:optimal_control}
The optimal control problem \eqref{Jintro}--\eqref{cc} has a unique solution $(\ousf(\ozsf), \ozsf)  $ 
$\in$ $ \mathbb{W} \times \Zad$. 
\end{theorem}
\begin{proof}
Using the operator $\mathbf{S}$, problem \eqref{Jintro}--\eqref{cc} reduces to: Minimize
\begin{equation}
\label{f_opt}
f(\zsf): = \frac12  \| \mathbf{S} \zsf - \usf_d \|^2_{L^2(Q)} + \frac\mu2 \|\zsf \|^2_{L^2(Q)},
\end{equation}
over $\Zad$. Since $\mu >0$ the strict convexity of $f$ is immediate. $\mathbf{S}$ is continuous, so $f$ is weakly lower semicontinuous. $\Zad$ is 
weakly sequentially compact. The direct method of the calculus of variations \cite[Theorem 1.15]{MR1201152} allows us to conclude.
\end{proof}

\subsection{Formal Lagrangian formulation}
\label{subsub:lagrange}
We now formally derive first-order necessary and sufficient optimality conditions for the control problem \eqref{Jintro}--\eqref{cc}. We 
proceed via the \emph{Lagrangian approach} described in \cite[\S3.1]{Tbook}. We must emphasize that, although these computations are merely formal, they are quite insightful as they allow us to determine what is the correct form of the optimality conditions with a simple and intuitive procedure.

Let $\psf$ denote the adjoint variable, the Lagrangian $\mathbb{L}: \mathbb{W} \times \Zad \times \mathbb{W} \to \R$ is
\[
  \mathbb{L}(\usf,\zsf,\psf) \mapsto J(\usf,\zsf)  - \int_{Q} \left( \partial^{\gamma}_t \usf + \mathcal{L}^s \usf -\fsf -\zsf \right) \psf.
\]
We expect the following necessary and sufficient optimality conditions \cite[\S3.1]{Tbook}:
\begin{align}
\label{Dp}
\int_Q D_{\psf} \mathbb{L} (\ousf,\ozsf,\opsf) h &= 0 \quad \forall h \in \mathbb{W},
\\
\label{Du}
\int_Q D_{\usf} \mathbb{L} (\ousf,\ozsf,\opsf) h &= 0 \quad \forall h \in \mathbb{W}, 
\textrm{ with } h(0) = 0,
\\
\label{Dz}
 \int_Q D_{\zsf} \mathbb{L} (\ousf,\ozsf,\opsf) (\zsf-\ozsf) &\geq 0 \quad \forall \zsf \in \Zad.
\end{align}
We start with a formal computation which uses the 
integration by parts formula \eqref{fractional_by_parts}:
\begin{align*}
\int_{Q} \left( \partial^{\gamma}_t \ousf + \mathcal{L}^s \ousf -\fsf - \ozsf \right) \opsf 
& = \int_{Q} \left( \partial^{\gamma}_{T-t} \opsf + \mathcal{L}^{s} \opsf \right) \ousf - \int_{Q} \left( \fsf + \zsf \right) \opsf 
\\ & - \int_{\Omega} \ousf(0) (I_{{T-t}}^{1-\gamma}\opsf)(0) + \int_{\Omega} \opsf(T) (I_{t}^{1-\gamma}\ousf)(T).
\end{align*}
Based on the previous computation, we rewrite expression \eqref{Du} as follows:
\begin{equation}
\label{eq:derivLagrangian}
\begin{aligned}
 \int_Q D_{\usf} \mathbb{L} (\ousf,\ozsf,\opsf)  h  &= (\ousf - \usf_d,h)_{L^2(Q)} 
- (\partial^{\gamma}_{T-t} \opsf + \mathcal{L}^{s} \opsf ,h)_{L^2(Q)} \\
&- (\opsf(T) , (I_{t}^{1-\gamma} h)(T))_{L^2(\Omega)} = 0,
\end{aligned}
\end{equation}
for all $h \in \mathbb{W}$ such that $h(0) = 0$. 

Let $\phi \in C_0^\infty(0,T)$, $\psi \in C_0^\infty(\Omega)$ be arbitrary and define $h = \psi \varphi$ 
where $\varphi$ solves the Abel equation
$
 (I_t^{1-\gamma} \varphi)(t) = \phi(t).
$
This is possible because of the unique solvability of the Abel equation given in \cite[Theorem 2.1, \S2.2]{Samko}. 
Notice that with this definition $(I_t^{1-\gamma}h)(T)=0$. Using this particular choice of $h$ in \eqref{eq:derivLagrangian} yields
\[
  \int_0^T \varphi(t)\left(\partial_{T-t}^\gamma \opsf + \mathcal{L}^s \opsf - ( \ousf - \usf_d), \psi \right)_{L^2(\Omega)} \diff t = 0.
\]
Owing to the results of \cite[Theorem 13.2, Theorem 13.5]{Samko}, the range of the fractional integral $I_t^{1-\gamma}$ contains all smooth functions. In other words the relation above must hold for all smooth and compactly supported $\varphi$, which implies
\begin{equation}
\label{equation_adjoint}
 \partial^{\gamma}_{T-t} \opsf + \mathcal{L}^s \opsf = \ousf - \usf_d.
\end{equation}
It remains to obtain a terminal condition for $\opsf$. To do so, we notice that we have
\[
  \big( \opsf(T), (I_t^{1-\gamma} h)(T) \big)_{L^2(\Omega)} = 0.
\]
If we were allowed to set $h$ constant in time this would yield $\opsf(T) = 0$. However, since $h(0) = 0$, the only admissible and constant in time function is $h \equiv 0$. To circumvent this we set $h = \ell_\epsilon(t)\chi$ with $\chi\in C_0^\infty(\Omega)$ arbitrary and 
$\ell_\epsilon(t) \in C^{0,1}([0,T])$ given by
\[
  \ell_\epsilon(t) = \epsilon^{-\gamma} T^{-\gamma} t^\gamma, \quad 0 < t \leq \epsilon T, \qquad
  \ell_\epsilon(t) = 1, \quad \epsilon T < t \leq T.
\]
This particular choice of $h$ yields
\[
  \left( \opsf(T), \chi \right)_{L^2(\Omega)} (I_t^{1-\gamma}\ell_\epsilon)(T) = 0.
\]
To conclude, it remains to notice that $\lim_{\epsilon\to0}(I_t^{1-\gamma} \ell_\epsilon)(T) =(I_t^{1-\gamma} 1 )(T)>0$.
Collecting the derived equations, our formal argument yields the following strong 
system for the adjoint variable $\psf$.

\begin{definition}[fractional adjoint state]
\label{def:fractional_adjoint} 
The solution $\psf = \psf(\zsf) \in \mathbb{W}$ of
\begin{equation}
\label{fractional_adjoint}
  \partial^{\gamma}_{T-t} \psf + \mathcal{L}^{s} \psf = \usf - \usf_d \ \text{in } Q, \quad
  \psf(T) = 0 \ \text{in } \Omega,
\end{equation}
for $\zsf \in L^2(0,T;\Hsd)$, is called the fractional adjoint state associated to $\usf = \usf(\zsf)$.
\end{definition}

\begin{remark}[$\gamma \to 1$] \rm
\label{RE:militing_case2}
For $g \in W^1_1(0,T)$, $\partial_{T-t}^{\gamma} g \rightarrow - \partial_t g$ as $\gamma \uparrow 1$ and we recover
\[
  -\partial_t \psf + \mathcal{L}^{s} \psf =  \usf - \usf_d \ \text{in } Q, \quad
  \psf(T) = 0 \ \text{in } \Omega,
\]
a standard backwards parabolic problem with terminal condition.
\end{remark}

Well-posedness of \eqref{fractional_adjoint} follows from a change of 
variables. If $\psi:[0,T] \to \mathbb R$ define $\tilde \psi(t) := \psi(T-t)$ and notice that $\tilde \psi'(t) = -\psi'(T-t)$. Therefore, with $c_\gamma = \Gamma(1-\gamma)$,
\[
  c_\gamma \partial_t^\gamma \tilde \psi(T-t) = \int_{0}^{T-t} \frac{-\psi'(T-\xi)}{( (T-t) -\xi)^{\gamma}} \diff \xi
  = - \int_{t}^{T} \frac{\psi'(\mu)}{(\mu - t)^{\gamma}} \diff \mu =  c_\gamma \partial_{T-t}^{\gamma} \psi(t).
\]
As a consequence, the backwards in time problem \eqref{fractional_adjoint} with a right Caputo fractional derivative can be equivalently written as a forward in time problem with a left Caputo fractional derivative as \eqref{fractional_heat}. The well-posedness of \eqref{fractional_adjoint} then follows from \S\ref{sub:stateequation}.

We conclude this formal analysis with the following variational inequality:
\begin{equation}
\label{viz}
\left( \mu \ozsf + \opsf, \zsf - \ozsf \right)_{L^2(Q)} \geq 0, \qquad \forall \zsf \in \Zad,
\end{equation}
which follows from \eqref{Dz}.

\begin{remark}[Lagrangian approach]\rm
Although formal, this approach is systematic and useful to derive optimality conditions of a control problem, specially in our case, where the state equation \eqref{fractional_heat} involves fractional derivatives in time and space.
\end{remark}

\subsection{Optimality conditions}
\label{sub:optim}
We begin with a classical result.

\begin{lemma}[variational inequality]
Let $f$ be defined by \eqref{f_opt}. The function $\ozsf \in \Zad$ minimizes the functional $f$ if and only if
\begin{equation}
\label{vi}
  (f'(\ozsf),\zsf - \ozsf )_{L^2(Q)} \geq 0,
\end{equation}
for every $\zsf \in \Zad$.
\end{lemma}
\begin{proof}
See \cite[Lemma 2.21]{Tbook}.
\end{proof}

To derive first-order optimality conditions, we need the following result.

\begin{lemma}[auxiliary result I]
\label{le:identity}
Let $\ozsf$ denote the optimal control given by Theorem~\ref{TH:optimal_control} and $\ousf = \mathbf{S}\ozsf$.
Then, for every $\zsf \in \Zad$, we have
\begin{equation}
\label{identity}
  (\ousf - \usf_d,\usf - \ousf)_{ L^2(Q)} =  ( \opsf, \zsf - \ozsf)_{L^2(Q)},
\end{equation}
where $\usf = \mathbf{S}\zsf \in \mathbb{W}$ and $\psf = \psf(\zsf) \in \mathbb{W}$ solve problems \eqref{fractional_heat} and \eqref{fractional_adjoint}, respectively.
\end{lemma}
\begin{proof}
Define $\phi:= \usf - \ousf \in \mathbb{W}$ and notice that $\phi(0) = 0$ in $\Omega$. Moreover
\begin{equation}
\label{phi}
    \langle \partial^{\gamma}_t \phi + \mathcal{L}^s \phi, w \rangle = (\zsf - \ozsf, w)_{L^2(\Omega)} 
    \qquad \forall w \in \Hs, \ \text{a.e. } (0,T).
\end{equation}
Since $\opsf \in \mathbb{W}$ setting $w = \opsf(t)$ in \eqref{phi} and integrating over time yields
\[
 ( \opsf, \zsf - \ozsf)_{L^2(Q)} 
 = \int_{0}^T \langle \partial^{\gamma}_t \phi  + \mathcal{L}^s \phi , \opsf \rangle.
\]
Lemma~\ref{lem:ibp} and the fact that the operator $\mathcal{L}^s$ is self adjoint allow us to write
\[
 ( \opsf, \zsf - \ozsf)_{L^2(Q)} 
 = \int_{0}^{T} \langle \partial^{\gamma}_{T-t} \psf + \mathcal{L}^s \opsf , \phi \rangle,
\]
where we used the terminal and initial conditions $\opsf(T)=0$ and $\phi(0)=0$, respectively, which are well defined in view of
Remark~\ref{rm:cont_t}. On the other hand, setting $\phi$ as test function in the weak version of \eqref{fractional_adjoint} and integrating in time yields
\[
 \int_{0}^T \langle \partial^{\gamma}_{T-t} \opsf  + \mathcal{L}^s \opsf , \phi \rangle = (\ousf-\usf_d,\phi)_{ L^2(Q)}.
\]
The desired identity \eqref{identity} follows easily from the derived expressions.
\end{proof}

We now prove necessary and sufficient optimality conditions for \eqref{Jintro}--\eqref{cc}.

\begin{theorem}[first-order optimality conditions]
\label{TH:fractional_fooc}
$\ozsf \in \Zad$ is the optimal control of problem \eqref{Jintro}--\eqref{cc} if and only if it solves
\eqref{viz}, where $\opsf = \opsf(\ozsf)$ solves \eqref{fractional_adjoint}.
\end{theorem}
\begin{proof}
We recall the control to state operator $\mathbf{S}:L^2(0,T,\Hsd) \rightarrow \mathbb{W}$ defined by $\mathbf{S}(\zsf) = \usf(\zsf)$, where $\usf(\zsf) \in \mathbb{W}$ solves problem \eqref{fractional_heat}. Next, we write $\mathbf{S}(\zsf) = \mathbf{S}_0(\zsf) + \psi_0$, where $\mathbf{S}_0(\zsf)$ denotes the solution to \eqref{fractional_heat} with $\fsf=0$ and $\usf_0 = 0$, while $\psi_0$ solves \eqref{fractional_heat} with $\zsf = 0$.
Since $\mathbf{S}_0$ is linear, in our setting the variational inequality \eqref{vi} reads 
\[
( \ousf - \usf_{d}, \mathbf{S}_0(\zsf - \ozsf) )_{L^2(Q)} + \mu ( \ozsf,  \zsf - \ozsf )_{L^2(Q)} \geq 0.
\]
Since $\mathbf{S}_0(\zsf - \ozsf) = \mathbf{S}_0 \zsf + \psi_0 - (\psi_0 + \mathbf{S}_0 \ozsf) = \usf(\zsf) - \ousf$,
the previous expression becomes
\[
( \ousf - \usf_{d}, \usf(\zsf) - \ousf )_{L^2(Q)} + \mu ( \ozsf,  \zsf - \ozsf )_{L^2(Q)} \geq 0.
\]
Using identity \eqref{identity} of Lemma \ref{le:identity}, we arrive at
\begin{align*}
( \opsf,  \zsf - \ozsf )_{L^2(Q)} + \mu ( \ozsf,  \zsf - \ozsf )_{L^2(Q)} \geq 0,
\end{align*}
which is \eqref{viz} and concludes the proof.
\end{proof}

\subsection{Regularity of the optimal control}
\label{sub:regularity}
Since we shall be concerned with approximating the solution to the control problem \eqref{Jintro}--\eqref{cc}, it is essential to study its regularity. Here, on the basis of a bootstrap argument, we obtain such results.

In what follows we will, without explicit mention, make the following regularity assumption concerning the domain $\Omega$:
\begin{equation}
\label{Omega_regular}
 \| w \|_{H^2(\Omega)} \lesssim \| \mathcal{L} w \|_{L^2(\Omega)}, \quad \forall w \in H^2(\Omega) \cap H^1_0(\Omega), 
\end{equation}
which is valid, for instance, if the domain $\Omega$ is convex \cite{Grisvard}. In addition, we will need the following assumption on $\asf$ and $\bsf$ defining the set $\Zad$:
\begin{equation}
\label{ab_condition}
\asf \leq 0 \leq \bsf \textrm{ on } \partial \Omega \times (0,T).
\end{equation}

\begin{theorem}[regularity of $\ozsf$]
\label{th:regofz}
Let $\gamma = 1$.
Assume that, for every $\epsilon>0$, we have $\fsf, \usf_d \in L^2(0,T;\mathbb{H}^{1-\epsilon}(\Omega))$ and that $\usf_0 \in \mathbb{H}^{1-\epsilon}(\Omega)$. If $\asf,\bsf \in H^1(Q)$ and \eqref{ab_condition} holds
then $\ozsf$, the solution to the optimal control problem \eqref{Jintro}--\eqref{cc}, satisfies
\begin{align*}
  \| \ozsf \|_{L^2(0,T;H^1(\Omega))} + \| \ozsf \|_{H^1(0,T;L^2(\Omega))} &\lesssim 
  \| \fsf \|_{L^2(0,T;\mathbb{H}^{1-\epsilon}(\Omega))} \\
  &+ \| \usf_d \|_{L^2(0,T;\mathbb{H}^{1-\epsilon}(\Omega))} + \| \usf_0 \|_{\mathbb{H}^{1-\epsilon}(\Omega)},
\end{align*}
where the hidden constant does not depend on the problem data. Moreover, $\ozsf \in L^2(0,T;H_0^1(\Omega))$.
\end{theorem}
\begin{proof}
The proof is based on a bootstrap argument as in \cite[Lemma 4.9]{AO}, so we merely sketch it.
By assumption, the right hand side of the state equation \eqref{fractional_heat} satisfies $\fsf + \ozsf \in L^2(Q)$, while the initial condition satisfies $\usf_0 \in \Hs$. Standard regularity arguments yield that the solution verifies $\ousf \in H^1(0,T;L^2(\Omega)) \cap L^\infty(0,T;\Hs)$. In addition, by writing the problem as
\[
  \mathcal{L}^s \ousf = \fsf + \ozsf - \partial_t \ousf \in L^2(Q),
\]
we realize that $\ousf \in L^2(0,T; \mathbb{H}^{2s}(\Omega))$. The right hand side of \eqref{fractional_adjoint} verifies $\ousf - \usf_d \in L^2(Q)$ so $\opsf$ has the same regularity that $\ousf$ possesses. From \cite[\S3.6.3]{Tbook} we have
\begin{equation}
\label{eq:projformula}
  \ozsf = \max \left\{ \asf, \min\left\{ \bsf, -\frac1\mu \opsf\right\} \right\}.
\end{equation}
This immediately yields $\ozsf \in H^1(0,T;L^2(\Omega))$.

To obtain the claimed space regularity we recall that $\opsf \in L^2(0,T;\mathbb{H}^{2s}(\Omega))$ and consider two cases:

\noindent \framebox{1} $s \in [\tfrac12,1)$: Since $\opsf \in L^2(0,T;H^1_0(\Omega))$, formula \eqref{eq:projformula} yields $\ozsf \in L^2(0,T;H^1_0(\Omega))$. Notice that assumption \eqref{ab_condition} is needed here to preserve the boundary values.

\noindent \framebox{2} $s \in (0,\tfrac12)$: We now begin the bootstrapping argument. A nonlinear operator interpolation argument as in \cite[Lemma 4.9]{AO} yields that $\ozsf \in L^2(0,T;H^{2s}(\Omega))$. From this and condition \eqref{ab_condition} we conclude that $\ozsf \in L^2(0,T; \mathbb{H}^{2s}(\Omega))$. Define $w_3 = \mathcal{L}^s \ousf$ and notice that, since $\fsf + \ozsf \in L^2(0,T;\mathbb{H}^{2s}(\Omega))$, $w_3$ solves $\partial_t w_3 + \mathcal{L}^s w_3 = \mathcal{L}^s(\fsf + \ozsf) \in L^2(Q)$ with initial condition $w_3(0) = \mathcal{L}^s \usf_0 \in L^2(\Omega)$. Therefore $w_3 \in L^2(0,T;\Hs)$, which implies $\ousf = \mathcal{L}^{-s} w_3 \in L^2(0,T;\mathbb{H}^{3s}(\Omega))$. Define now $q_3 = \calLs \opsf$ and notice that the same arguments yield that $q_3 \in L^2(0,T;\Hs)$ and, therefore $\opsf \in L^2(0,T;\mathbb{H}^{3s}(\Omega))$. We consider, again, two cases:

\noindent \framebox{2.1} $s \in [\tfrac13, \tfrac12)$: As in step \framebox{1} we have that $\ozsf \in L^2(0,T;H_0^1(\Omega))$.

\noindent \framebox{2.2} $s \in (0, \tfrac13)$: Nonlinear operator interpolation and \eqref{ab_condition}, again, give us that $\ozsf \in L^2(0,T;\mathbb{H}^{3s}(\Omega))$. Define now $w_4 = \calL^{s/2} w_3$, which solves $\partial_t w_4 + \calLs w_4 = \calL^{3s/2}(\fsf + \ozsf) \in L^2(Q)$ with $w_4(0) = \calL^{3s/2}\usf_0 \in L^2(\Omega)$. This again yields that $\ousf, \opsf \in L^2(0,T;\mathbb{H}^{4s}(\Omega))$. We consider, one more time, two cases:

\noindent \framebox{2.2.1} $s \in [\frac14, \tfrac13)$: In this case $\ozsf \in L^2(0,T;H_0^1(\Omega))$.

\noindent \framebox{2.2.2} $s \in (0,\tfrac14)$: Define $w_5 = \calL^{s/2} w_4$ and argue as before.

Proceeding in this way we can conclude, after a finite number of steps, that for any $s \in (0,\tfrac{1}{2})$ we have $\ozsf \in L^2(0,T;H_0^1(\Omega))$. This concludes the proof.
\end{proof}

\begin{remark}[regularity of $\ousf$ and $\opsf$] \rm
\label{rk:reg_states}
Notice that while proving Theorem~\ref{th:regofz} we have also shown that $\ousf,\opsf \in H^1(0,T;L^2(\Omega)) \cap L^2(0,T;H^1_0(\Omega))$.
\end{remark}

\subsection{The extended control problem}
\label{sub:extcontrol}
To circumvent the nonlocality of the operator $\mathcal{L}^s$ in problem \eqref{Jintro}--\eqref{cc} we realize it using the Caffarelli-Silvestre extension. In what follows we consider the \emph{equivalent} problem: Find $\min\{ J(\tr \ue,\zsf) : \ue \in \mathbb{V}, \zsf \in \Zad \}$ subject to the \emph{extended state equation}: Find $\ue \in \mathbb{V}$ such that $\tr \ue(0) = \usf_0$ in $\Omega$ and, for a.e.~$t \in (0,T)$,
\begin{equation}
\label{state_equation_extended}
\langle \tr  \partial_t^{\gamma} \ue, \tr \phi \rangle + a(\ue,\phi) 
  = \langle \fsf + \zsf, \tr \phi \rangle \qquad \forall \phi \in \HL(y^{\alpha},\C).
\end{equation}

To describe the optimality conditions we introduce the \emph{extended adjoint problem}: Find $\pe \in \mathbb{V}$ such that $\tr \pe(T)=0$ in $\Omega$ and, for a.e.~$t \in (0,T)$,
\begin{equation}
\label{extended_adjoint}
\langle \tr  \partial_{T-t}^{\gamma} \pe, \tr \phi \rangle + a(\pe,\phi)  =  ( \tr \ue - \usf_d , \tr \phi )_{L^2(\Omega)}, \quad 
\forall \phi \in \HL(y^{\alpha},\C).
\end{equation}
The optimality conditions in this setting now read as follows: the pair $(\bar{\ue}(\ozsf),\ozsf) \in \mathbb{V}\times \Zad$ is optimal if and only if $\bar{\ue}(\ozsf)$ solves \eqref{state_equation_extended} and
\begin{equation}
\label{op_extended}
  (\tr \bar{\pe} + \mu \ozsf , \zsf - \ozsf )_{L^2(Q)} \geq 0 \quad \forall \zsf \in \Zad,
\end{equation}
where $\bar\pe = \bar\pe(\ozsf) \in \mathbb{V}$ solves \eqref{extended_adjoint}.

\section{A truncated optimal control problem}
\label{sec:control_truncated}

The state equation \eqref{state_equation_extended} is posed on the infinite cylinder $\C = \Omega \times (0,\infty)$,
therefore it cannot be directly approximated with finite element-like techniques. The first step towards discretization is to truncate $\C$ to a bounded cylinder $\C_{\Y}= \Omega \times (0,\Y)$, which is possible because the energy decreases exponentially in $\Y$; see \cite[Proposition 4.1]{NOS3} for details. 

\begin{proposition}[exponential decay]
\label{pro:energyYinf}
If, for a given $\gamma \in (0,1]$ and $s \in (0,1)$, $\ue = \ue(\zsf) \in \mathbb{V}$ solves \eqref{state_equation_extended}, then for every $\Y > 1$ we have
\begin{equation}
\label{energyYinf}
\|\nabla \ue \|_{L^2( 0,T; L^2(y^\alpha,\Omega \times (\Y,\infty))) } \lesssim e^{-\sqrt{\lambda_1} \Y/2}
\Lambda_{\gamma}(\usf_0,\fsf + \zsf),
\end{equation}
where $\Lambda_{\gamma}$ is defined in \eqref{Lambda}.
\end{proposition}

Proposition~\ref{pro:energyYinf} motivates a \emph{truncated control problem} as follows. We first define 
\begin{align*}
\HL(y^{\alpha},\C_\Y) &= \left\{ w \in H^1(y^\alpha,\C_\Y): w = 0 \text{ on } \partial_L \C_\Y \cup \Omega \times \{ \Y\} \right\},\\
\mathbb{V}_{\Y} &= \{ w \in L^2(0,T;\HL(y^{\alpha},\C_{\Y})): \partial_t^{\gamma} \tr w  \in L^2(0,T;\Hsd)\}, 
\end{align*}
and, for $w, \phi \in \HL(y^{\alpha},\C_\Y)$, the bilinear form
\begin{equation}
\label{a_Y}
  a_\Y(w,\phi) =   \frac{1}{d_s}\int_{\C_\Y} y^{\alpha} \left( \mathbf{A}(x') \nabla w \cdot \nabla \phi
 + c(x') w \phi \right) \diff x' \diff y.
\end{equation}
We define the \emph{truncated control problem} as: Find $\min \{J(\tr v,\rsf): v \in \mathbb{V}_\Y, \rsf \in \Zad\}$, subject to the \emph{truncated state equation}: Find $v \in \mathbb{V}_{\Y}$ with $\tr v(0) = \usf_0$ in $\Omega$ and
\begin{equation}
\label{state_equation_truncated}
\langle \tr  \partial_t^{\gamma} v, \tr \phi \rangle + a_{\Y}(v,\phi)   = \langle \fsf + \rsf, \tr \phi \rangle, 
\quad \forall \phi \in \HL(y^\alpha,\C_\Y),
\quad a.e.~t \in (0,T).
\end{equation}

As an instrument we define $\mathcal{H}_\alpha: \Hs \rightarrow \HL(y^{\alpha},\C_{\Y})$, the $\alpha$-harmonic extension to $\C_\Y$, \ie if $\wsf \in \Hs$, then $w = \mathcal{H}_\alpha \wsf$ solves
\begin{equation}
\label{eq:mathcalH}
\DIV(y^{\alpha} \nabla w) = 0 \text{ in } \C_\Y, 
\quad
w = 0 \text{ on } \partial_L \C_\Y \cup \Omega \times \{\Y\} , 
\quad
w = \wsf \text{ on } \Omega \times \{0\}.
\end{equation}

\begin{remark}[initial datum]\rm
\label{rm:initial_datum}
The initial datum $\usf_0$ of \eqref{fractional_heat} determines $v(0)$ only on $\Omega \times \{ 0\}$ in a trace sense. We thus define $v(0) = \mathcal{H}_\alpha \usf_0$. Remark 3.4 in \cite{NOS} provides the estimate $\| \nabla v(0) \|_{L^2(y^{\alpha},\C_\Y)} \lesssim \| \usf_0 \|_{\Hs}$.
\end{remark}

Let us now provide, for $\gamma=1$, an energy estimate that will be useful to derive an $L^2(Q)$ error estimate for the fully-discrete scheme of \S\ref{sec:fully_scheme}.

\begin{theorem}[energy estimate: $\gamma = 1$] 
\label{TH:stab_v_1}
Let $s \in (0,1)$ and $\gamma = 1$ and denote by $v \in \V_{\Y}$ the solution to \eqref{state_equation_truncated}. If $\fsf$ and $\rsf$ belong to $L^2(Q)$ and $\usf_0 \in \Hs$, then
\begin{equation}
\label{eq:energy_v}
\| \tr \partial_t v \|_{L^2(Q)} + \| \nabla v \|_{L^\infty(0,T;L^2(y^{\alpha},\C_\Y))} \lesssim
\| \fsf + \rsf \|_{L^2(Q)} + \| \usf_0\|_{\Hs},
\end{equation}
where the hidden constant does not depend on $v$ nor the problem data.
\end{theorem}
\begin{proof}
Set $\phi = \partial_t v$ in \eqref{state_equation_truncated}, integrate over time and use 
the estimate of Remark~\ref{rm:initial_datum}: $\| \nabla v(0) \|_{L^2(y^{\alpha},\C_\Y)} \lesssim \| \usf_0 \|_{\Hs}$.
\end{proof}

As in \S\ref{sub:optim} we introduce the \emph{truncated adjoint problem}: Find $p \in \mathbb{V}_\Y$ such that $\tr p(T)=0$ and, for a.e.~$t \in (0,T)$,
\begin{equation}
\label{truncated_adjoint}
\langle \tr  \partial_{T-t}^{\gamma} p, \tr \phi \rangle 
+ a_{\Y}(p,\phi)  =  \langle \tr v - \usf_d , \tr \phi \rangle,
\quad \forall \phi \in \HL(y^{\alpha},\C_{\Y}).
\end{equation}

The same arguments provided in Theorem~\ref{TH:fractional_fooc} allow us to conclude that the pair $(\bar{v}(\orsf), \orsf) \in$ $\mathbb{V}_{\Y} \times \Zad$ is optimal if and only if $\bar{v}(\orsf)$ solves \eqref{state_equation_truncated} and $\orsf$ satisfies
\begin{equation}
\label{op_truncated}
  ( \tr \bar{p} + \mu \orsf , \rsf - \orsf )_{L^2(Q)} \geq 0 \quad \forall \rsf \in \Zad,
\end{equation}
where $\bar p = \bar p (\orsf) \in \mathbb{V}_\Y$ solves \eqref{truncated_adjoint}.

The next result shows how $(\bar{v}(\orsf),\orsf)$ approximates $(\bar{\ue}(\ozsf),\ozsf)$.

\begin{lemma}[exponential convergence]
\label{LE:exp_convergence}
For every $\Y \geq 1$ we have
\begin{equation}
\label{control_exp}
\| \orsf - \ozsf \|_{L^2(Q)} \lesssim  e^{-\frac{\sqrt{\lambda_1}}{2}\Y} \!
(\Lambda_{\gamma}(\usf_0,\fsf)  + \| \orsf \|_{L^2(0,T;\Hsd)} 
+\| \usf_d \|_{L^2(0,T;\Hsd)} ),
\end{equation}
and 
\begin{equation}
\label{state_exp}
\begin{aligned}
  \| \bar{\ue} - \bar{v} \|_{L^2(0,T;\HLn(y^\alpha,\C))}  &\lesssim e^{-\frac{\sqrt{\lambda_1}}{2}\Y}  ( \Lambda_{\gamma}(\usf_0,\fsf) \\
  &+ \| \orsf \|_{L^2(0,T;\Hsd)} +  \| \usf_d \|_{L^2(0,T;\Hsd)} 
),
\end{aligned}
\end{equation}
where $\Lambda_{\gamma}$ is defined in \eqref{Lambda}.
\end{lemma}

\begin{proof} We proceed in four steps:

\noindent \boxed{1} Set $\zsf = \orsf \in \Zad$ and $\rsf = \ozsf \in \Zad$ in the variational inequalities \eqref{op_extended} and \eqref{op_truncated}, respectively, and add the obtained inequalities to arrive at
\begin{align*}	
  \mu \| \orsf - \ozsf \|^2_{L^2(Q)}  & \leq ( \tr (\bar{\pe} -  \bar{p}), \orsf - \ozsf )_{L^2(Q)} \\
   &= ( \tr (\bar{\pe} -  \pe(\orsf)), \orsf - \ozsf )_{L^2(Q)}
  + ( \tr (\pe(\orsf)-\bar{p}), \orsf - \ozsf )_{L^2(Q)}.
\end{align*}

\noindent \boxed{2} Consider $( \tr (\bar{\pe} -  \pe(\orsf)), \orsf - \ozsf )_{L^2(Q)}$. Define $\psi :=  \ope - \pe(\orsf) \in \mathbb{V}$ and observe that $\tr \psi(T) = 0$ and, for all $\phi_p \in \HL(y^{\alpha},\C)$, we have
\[
\int_{0}^T  \left( \langle \tr \partial_{T-t}^{\gamma}  \psi, \tr \phi_p  \rangle + a(\psi,\phi_p) \right) \diff t
= \int_{0}^T    \left( \tr (\oue - \ue(\orsf)),\tr \phi_p \right)_{L^2(\Omega)} \diff t.
\]
Analogously, define $\varphi := \oue - \ue(\orsf) \in \mathbb{V}$, which satisfies $\tr \varphi(0) = 0$ and
\[
 \int_{0}^T \left( \langle \tr \partial_{t}^{\gamma}  \varphi, \tr \phi_u  \rangle + a(\varphi,\phi_u) \right) \diff t
= \int_{0}^T \left( \ozsf - \orsf,\tr \phi_u \right)_{L^2(\Omega)} \diff t, ~~\forall \phi_u \in \HL(y^{\alpha},\C).
\]
Set $\phi_p = \varphi$, $\phi_u = \psi$ and apply Lemma~\ref{lem:ibp}. Since $\tr \psi (T) = \tr \varphi (0) = 0$ we get
\[
 ( \tr (\ope - \pe(\orsf)), \orsf - \ozsf )_{L^2(Q)} = - \| \tr (\oue - \ue(\orsf)) \|^2_{L^2(Q)} \leq 0.
\]

\noindent \boxed{3} The previous step shows that proving \eqref{control_exp} reduces to obtaining a bound for $\| \tr (\pe(\bar{\rsf}) - \bar p ) \|_{L^2(Q)}$ and \cite[Lemma 4.3]{NOS3} yields such a bound.

\noindent \boxed{4} In order to prove \eqref{state_exp} we write
\[
  \bar{\ue}(\ozsf) - \bar{v}(\orsf) = \left( \bar{\ue}(\ozsf) - \ue(\orsf) \right) + \left( \ue(\orsf) - \bar{v}(\orsf) \right).
\]
The first term satisfies \eqref{heat_harmonic_extension_weak} with right hand side $\ozsf - \orsf$ so that by \eqref{control_exp} this term is bounded. For the second term we again apply \cite[Lemma 4.3]{NOS3}.
\end{proof}

\begin{remark}[regularity of $\orsf$ vs.~$\ozsf$] \rm
\label{rm:regularity_control_r}
In Theorem~\ref{th:regofz} we studied the regularity of $\ozsf$. The techniques of \cite[Remark 4.4]{NOS3} allow us to transfer these results to $\orsf$, the solution of the truncated optimal control problem. In a similar fashion, we can establish the regularity results of Remark~\ref{rk:reg_states} for $\tr \bar{v}$ and $\tr \bar{p}$. For brevity we skip the details.
\end{remark}

\section{Approximation of the state equation}
\label{sec:a_priori_state}

We recall the numerical approximation of the state equation \eqref{heat_harmonic_extension_weak} developed in \cite{NOS3}. The scheme employs first degree tensor product finite elements in space and finite differences in time. The latter is the backward Euler scheme for $\gamma = 1$ whereas, for $\gamma \in (0,1)$, it is the scheme of \cite{LLX:11,LinXu:07}, which was studied under appropriate time-regularity conditions on the solution $\ue$ in \cite{NOS3}. We also derive a novel $L^2(Q)$ a priori error estimate for the fully discrete approximation of the state equation \eqref{heat_harmonic_extension_weak} with $\gamma = 1$ and $s \in (0,1)$.

\subsection{Time discretization}
\label{sec:time_discretization}
Let $\K \in \mathbb{N}$ denote the number of time steps. Define the uniform time step $\tau = T/\K > 0$, and set $t_k = k \tau$ for $0 \leq k \leq \K$. We denote the time partition by
$\mathcal{T}:= \left\{ t_k \right\}_{k=0}^\K
$.
If $\phi \in C ( [0,T], {\Xcal} )$, we denote $\phi^k = \phi(t_k)$ and $\phi^{\tau}= \{ \phi^k\}_{k=0}^{\K}$. On such sequences we define the norms
\[
\| \phi^{\tau} \|_{\ell^{\infty}(\Xcal)} = \max_{0 \leq k \leq \K} \| \phi^k\|_{\Xcal},
\qquad \| \phi^{\tau} \|_{\ell^2(\Xcal)}^2 = \sum_{k=1}^\K \tau \| \phi^k\|_{\Xcal}^2.
\]
Over sequences $\phi^{\tau} \subset \Xcal$ we define the discrete time derivative $\delta^1$ by
\begin{equation}
\label{1_discrete}
  \delta^1 \phi^{k+1} = \tau^{-1} (\phi^{k+1} - \phi^{k}), \quad k = 0,\ldots, \K -1 .
\end{equation}
As in \cite[\S3.2]{NOS3} we also define, for $\gamma\in(0,1)$, the discrete fractional derivative $\delta^\gamma$ as
\begin{equation}
\label{gamma_discrete}
  \Gamma(2-\gamma)\delta^\gamma \phi^{k+1} := \sum_{j=0}^{k} \frac{  a_j }{ \tau^{\gamma-1} } \delta^1\phi^{k+1-j}
  = \frac{\phi^{k+1}}{\tau^\gamma}
  - \sum_{j=0}^{k-1} \frac{ a_j - a_{j+1} }{\tau^\gamma} \phi^{k-j} - \frac{a_k}{\tau^\gamma} \phi^0,
\end{equation}
where $a_j = (j+1)^{1-\gamma} - j^{1-\gamma}$ and provided the sum for $k = 0$ is defined to be zero.

We remark that, any sequence $\phi^\tau \subset \Xcal$ can be equivalently understood as a piecewise constant, in time, function $\phi \in L^\infty(0,T;\Xcal)$:
\begin{equation}
\label{eq:equivalence}
 \phi(t) = \phi^k, \quad \forall t \in (t_{k-1}, t_k], \qquad k = 1,\cdots,\K.
\end{equation}
This identification will be very useful and, in what follows, we will use it repeatedly and without explicit mention.

\subsection{Space discretization}
\label{sec:space_discretization}

The space discretization is based on truncation and the finite element method. The truncation is as in \cite[Lemma 4.3]{NOS3}, which shows that truncating $\C$ to $\C_{\Y}$ induces an exponentially small error. Since we are now dealing with the bounded domain $\C_\Y$, we can discretize using finite elements.

The finite element discretization follows \cite[\S4]{NOS}.
Let $\T_{\Omega} = \{K\}$ be a conforming triangulation of $\Omega$ into cells $K$ (simplices or $n$-rectangles).
We denote by $\Tr_\Omega$ the collection of all conforming refinements of an original mesh $\T_{\Omega}^0$ and assume $\Tr_\Omega$ is shape regular \cite{CiarletBook}. If $\T_\Omega \in \Tr_\Omega$ we define $h_{\T_{\Omega}}= \max_{K \in \T_\Omega}h_K$. We define $\T_\Y$ to be a partition of $\C_\Y$ into cells of the form $T = K \times I$, where $K \in \T_\Omega$, and $I$ is an interval that comes from the partition $\{ y_m\}_{m=0}^{M}$ of $[0,\Y]$ defined by
\begin{equation}
\label{graded_mesh}
  y_m = \left( \frac{m}{M}\right)^{\zeta} \Y, \quad m=0,\dots,M,
\end{equation}
where $\zeta = \zeta(\alpha) > 3/(1-\alpha)> 1$. The set of all such triangulations is denoted by $\Tr$. Note that the following weak regularity condition is valid: there is a constant $\sigma$ such that, for all $\T_\Y \in \Tr$, if $T_1=K_1\times I_1,T_2=K_2\times I_2 \in \T_\Y$ have nonempty intersection, then
$h_{I_1}/h_{I_2} \leq \sigma$, where $h_I = |I|$; see \cite{NOS2,NOS}. For $\T_\Y \in \Tr$, we denote by $\N(\T_{\Y})$ the set of its nodes and $\Nin(\T_{\Y})$ the set of its interior and Neumann nodes, respectively. We also denote by $N = \# \, \Nin(\T_{\Y})$ the number of degrees of freedom of $\T_\Y$. We assume that $\#\T_\Omega \approx M^n$ so that $N\approx M^{n+1}$. 

The main motivation to consider elements as in \eqref{graded_mesh} is to compensate the rather singular behavior of $\ue$, solution to problem \eqref{heat_harmonic_extension_weak} as $y \approx 0^+$; see \cite{NOS3} for details. 

For $\T_{\Y} \in \Tr$ and $\Gamma_D = \partial_L \C_{\Y} \cup \Omega \times \{ \Y\}$ we define the finite element space
\[
  \V(\T_\Y) = \left\{
            W \in C( \bar{\C}_\Y): W|_T \in \mathcal{P}_1(K)\otimes \mathbb{P}_1(I) \ \forall T = K \times I \in \T_\Y, \
            W|_{\Gamma_D} = 0
          \right\},
\]
If $K$ is a simplex, then $\mathcal{P}_1(K)=\mathbb{P}_1(K)$, whereas if $K$ is a cube, then $\mathcal{P}_1(K)=\mathbb{Q}_1(K)$. We also define $\U(\T_{\Omega})=\tr \V(\T_{\Y})$, \ie a $\mathcal{P}_1$ finite element space over the mesh $\T_\Omega$.

\subsection{A fully discrete scheme}
\label{sec:fully_scheme}

The fully discrete scheme to solve \eqref{heat_alpha_extension} combines the space discretization of \S\ref{sec:space_discretization} with the time discretization of \S\ref{sec:time_discretization}. To define it, we first consider the weighted elliptic projector $G_{\T_\Y}$ studied in \cite[\S 4.3]{NOS3}:
\begin{equation}
 \label{elliptic_projection}
  w \in \HL(y^\alpha,\C_\Y): \qquad a_\Y \big( G_{\T_{\Y}}w , W \big)  = a_\Y(w, W), \quad \forall W \in \V(\T_{\Y}).
\end{equation}

The fully-discrete scheme computes $V_{\T_{\Y}}^\tau  \subset \V(\T_{\Y})$, an approximation of the solution to \eqref{state_equation_truncated}, with $\rsf = 0$, at each time step. We initialize the scheme by setting
\begin{equation}
\label{initial_data_discrete}
 V_{\T_{\Y}}^{0} = \mathcal I_{\T_\Omega}\usf_0 = G_{\T_\Y} \circ \mathcal{H}_\alpha \usf_0,
\end{equation}
where $\mathcal I_{\T_\Omega} = G_{\T_\Y} \circ \mathcal{H}_\alpha$ and $\mathcal{H}_\alpha$ is the $\alpha$-harmonic extension operator defined in \S\ref{sec:control_truncated}; notice that $\tr V_{\T_\Y}^0 = \tr G_{\T_\Y} v(0)$, where $v(0)$ solves \eqref{eq:mathcalH} with $\wsf = \usf_0$.

For $k=0,\dots,\K-1$, $V_{\T_{\Y}}^{k+1} \in \V(\T_{\Y})$ solves
\begin{equation}
\label{fully_beta}
 ( \delta^\gamma \tr V_{\T_{\Y}}^{k+1} , \tr W )_{L^2(\Omega)}  + a_\Y(V_{\T_{\Y}}^{k+1},W) = \left\langle \fsf^{k+1}, \tr W   \right\rangle,
 \quad
 \forall W \in \V(\T_\Y),
 \end{equation}
where $\delta^\gamma$ is defined by \eqref{gamma_discrete} for $\gamma \in (0,1)$ and by \eqref{1_discrete} for $\gamma = 1$ and $\fsf^{k+1}= \tau^{-1} \int_{t_k}^{t_{k+1}} \fsf \diff t$. An approximate solution of problem \eqref{fractional_heat} is $U_{\T_{\Omega}}^\tau  \subset \U(\T_\Omega)$ with
\begin{equation}
\label{discrete_Ufd}
 U_{\T_{\Omega}}^\tau = \tr V_{\T_{\Y}}^\tau.
\end{equation}

To present error estimates for scheme \eqref{initial_data_discrete}--\eqref{fully_beta}, for $\gamma \in (0,1)$, we introduce
\begin{align*}
\mathcal{A} = \mathcal{A}(\vsf,\gsf) &=  \| \vsf\|_{\Hs} + \| \gsf \|_{H^2(0,T;\Hsd)},\\
\mathcal{B} = \mathcal{B}(\vsf,\gsf) &= \| \vsf \|_{\mathbb{H}^{1+3s}(\Omega)} + \| \gsf|_{t=0} \|_{\mathbb{H}^{1+s}(\Omega)} 
  + \| \gsf \|_{W^{1}_{\infty}(0,T;\mathbb{H}^{1-(1-2\nu)s}(\Omega))},
\end{align*}
where $\nu>0$ is arbitrary. Theorem 5.3 in \cite{NOS3} provides the following error estimates for the scheme \eqref{initial_data_discrete}--\eqref{fully_beta} with $\gamma \in (0,1)$ and $s \in (0,1)$.

\begin{theorem}[error estimates: $s,\gamma \in (0,1)$]
\label{th:order_beta}
Let $\rsf = 0$ and $s,\gamma \in (0,1)$. If $v$ solves \eqref{state_equation_truncated}, $V_{\T_{\Y}}^\tau$ solves \eqref{initial_data_discrete}--\eqref{fully_beta}, $\mathcal{A}(\usf_0,\fsf)\mathcal{B}(\usf_0,\fsf)<\infty$ and $\T_{\Y}$ verifies \eqref{graded_mesh}, then
\begin{equation}\label{estimate_1}
[I^{1-\gamma}_t\| \tr (v^\tau - V_{\T_{\Y}}^\tau) \|_{L^2(\Omega)}^2(T)]^{\tfrac{1}{2}} \lesssim 
\tau^{\theta} \mathcal{A}(\usf_0,f) + |\log N|^{2s} N^\frac{-(1+s)}{n+1} \mathcal{B}(\usf_0,f),
\end{equation}
and
\begin{equation}\label{estimate_2}
\| v^\tau - V_{\T_{\Y}}^\tau \|_{\ell^{2}(\HLn(y^{\alpha},\C_\Y) )} \lesssim \tau^{\theta} \mathcal{A}(\usf_0,f) 
+ |\log N|^{s} N^\frac{-1}{n+1} \mathcal{B}(\usf_0,f),
\end{equation}
where $\theta \in (0,\tfrac{1}{2})$, and the hidden constant does not depend on 
$v$, $V_{\T_{\Y}}^{\tau}$ nor the problem data, but blows up as $\theta \uparrow \tfrac12$.
\end{theorem}

\subsection{$L^2(Q)$-error estimate: $s \in (0,1)$ and $\gamma = 1$}
\label{sec:fully_scheme_1}

We now derive a novel $L^2(Q)$-error estimate for \eqref{initial_data_discrete}--\eqref{fully_beta} with $\gamma = 1$, which is inspired by classical techniques developed, for instance, in \cite{BGRV,RHN:85}. To obtain it, we set $\rsf = 0$ and consider, as a technical instrument, a semi-discrete approximation to \eqref{state_equation_truncated}: Set $V^0 = \mathcal{H}_\alpha \usf_0$ and, for $k=0,\dots,\K-1$, compute $V^{k+1}  \in \HL(y^{\alpha},\C_\Y)$, that solves
\begin{equation}
\label{eq:semi_1}
  ( \delta^1 \tr V^{k+1} , \tr \phi )_{L^2(\Omega)}  + a_\Y(V^{k+1},\phi) = \left\langle \fsf^{k+1}, \tr \phi   \right\rangle,
  \quad \forall \phi \in \HL(y^{\alpha},\C_\Y),
\end{equation}
where $\delta^1$ is defined by \eqref{1_discrete}. We present the following stability result.

\begin{lemma}[stability]
\label{LM:stab_sd}
Let $V^{\tau} \subset \HL(y^{\alpha},\C_\Y)$ solve \eqref{eq:semi_1}. If $\fsf \in L^2(Q)$ and $\usf_0 \in \Hs$, then we have
\begin{equation}
\label{eq:stab_sd}
\| \tr \delta^1 V^{\tau} \|^2_{\ell^2(L^2(\Omega))} + \| \nabla V^{\tau} \|^2_{\ell^\infty(L^2(y^{\alpha},\C_{\Y}))} \lesssim \| \fsf \|^2_{\ell^2(L^2(\Omega))} + \| \usf_0\|^2_{\Hs},
\end{equation}
\end{lemma}
where the hidden constant does not depend on $V^{\tau}$ nor the problem data.
\begin{proof}
Set $\phi = V^{k+1} - V^{k}$ and use the estimate of Remark~\ref{rm:initial_datum}.
\end{proof}

Define the piecewise linear function $\hat{V} \in C^{0,1}([0,T];\HL(y^\alpha,\C_\Y))$ by $\hat{V}(0) = V^0$ 
\begin{equation}
\label{hatV}
\hat{V}(t) = V^{k} + (t-t_k)\delta^{1}V^{k+1}, \quad t \in (t_k,t_{k+1}],
\end{equation}
for $k = 0,\dots,\K-1$. Using this notation, we rewrite equation \eqref{eq:semi_1} as
\begin{equation}
\label{newsemi_1}
  ( \tr \partial_t \hat{V}(t), \tr \phi )_{L^2(\Omega)}  + a_\Y(V^\tau(t),\phi) = \left\langle \fsf^{\tau}(t), \tr \phi   \right\rangle, \quad \forall \phi \in \HL(y^{\alpha},\C_\Y),
\end{equation}
for a.e.~$t \in (0,T)$. We are now in position to derive an error estimate for \eqref{eq:semi_1}.

\begin{theorem}[semi-discrete error estimate: $\gamma = 1$]
\label{TH:order_1}
Let $v$ and $V^\tau$ solve \eqref{state_equation_truncated} and \eqref{eq:semi_1}, respectively. If $\fsf \in L^{\infty}(0,T;L^2(\Omega))$ and $\usf_0 \in \Hs$, then
\begin{equation}
\label{estimate_11}
  \| \tr (v - V^\tau) \|_{L^2(0,T;L^2(\Omega))}  \lesssim \tau \left( \| \fsf \|_{L^{\infty}(0,T;L^2(\Omega))} + \| \usf_0 \|_{\Hs}\right), 
\end{equation}
where the hidden constant does not depend on $v$, $V^{\tau}$ nor the problem data.
\end{theorem}
\begin{proof}
Define $\hat{e} = v - \hat{V}$ and $\bar{e} = v - V^{\tau}$. Set $\gamma = 1$ and $\rsf = 0$ in \eqref{state_equation_truncated} and then subtract from it \eqref{newsemi_1}. Integrating with respect to time the result we obtain
\begin{multline*}
( \tr \bar{e}(t), \tr \phi )_{L^2(\Omega)}  + a_\Y\left(\int_0^t \bar{e}(\xi) \diff \xi, \phi \right) = 
  \left\langle \int_0^t (\fsf(\xi) - \fsf^{\tau}(\xi)) \diff \xi, \tr \phi   \right\rangle
  \\
  + ( \tr ( \bar{e}(t) - \hat{e}(t) ), \tr \phi )_{L^2(\Omega)}, \quad \text{a.e. } t\in (0,T), \quad \forall \phi \in \HL(y^{\alpha},\C_\Y).
\end{multline*}
Set, for a.e.~$t\in (0,T)$, $\phi = \bar{e}(t) \in \HL(y^{\alpha},\C_\Y)$. Integrating over time once more yields
\begin{multline}
\label{eq:csm}
  \int_0^T \| \tr \bar{e}(t) \|^2_{L^2(\Omega)}\diff t 
  \leq \left| \int_0^T \left \langle \int_0^t (\fsf(\xi) - \fsf^{\tau}(\xi)) \diff \xi, \tr \bar{e}(t)   \right\rangle \diff t \right| \\
  + \left| \int_0^T ( \tr ( \bar{e}(t) - \hat{e}(t) ), \tr \bar{e}(t) )_{L^2(\Omega)} \diff t \right|,
\end{multline}
where we used that 
\[
  \int_0^T a_\Y \left( \int_0^t \bar{e}(\xi) \diff \xi, \bar{e}(t) \right) \diff t  = \frac{1}{2}  a_\Y \left( \int_0^T \bar{e}(t) \diff t,\int_0^T \bar{e}(t) \diff t \right) \geq 0.
\]

Notice that, since $\fsf^{k+1} = \tau^{-1}\int_{t_k}^{t_{k+1}} \fsf(t) \diff t$,
\[
  \int_0^{t_l} (\fsf(\xi) - \fsf^\tau(\xi) ) \diff \xi = \sum_{k=1}^{l} \int_{t_{k-1}}^{t_k} \left( \fsf(\xi) - \fsf^k \right) \diff \xi = 0.
\]
Consequently, if $t_l \leq t < t_{l+1}$, we have
\[
  \int_0^t (\fsf(\xi) - \fsf^\tau(\xi) ) \diff \xi = \int_{t_l}^{t} (\fsf(\xi) - \fsf^\tau(\xi) ) \diff \xi \lesssim \tau \| \fsf \|_{L^\infty(0,t)}.
\]
In conclusion, the first term on the right hand side of \eqref{eq:csm} can be bound by
\[
  \left| \int_0^T \left\langle \int_0^t (\fsf(\xi) - \fsf^{\tau}(\xi)) \diff s, \tr \bar{e}(t)   \right\rangle \diff t \right| \lesssim
  \tau^2 \| \fsf \|_{L^\infty(0,T,L^2(\Omega))}^2 + \frac14 \| \tr \bar{e} \|_{L^2(Q)}^2.
\]

Since, on $(t_k,t_{k+1}]$, we have that $|\bar{e}(t) - \hat{e}(t)| \leq \tau |\delta^1 V^{k+1}|$, estimate \eqref{eq:stab_sd} yields
\[
  \int_0^T \| \tr ( \hat{e}(t) -\bar{e}(t) )  \|^2_{L^2(\Omega)}\diff t \leq
  \tau^2 \left\| \delta^1 V^{\tau} \right\|_{\ell^2(L^2(\Omega))}^2
  \lesssim \tau^2 \left( \| \fsf^{\tau}\|^2_{\ell^2(L^2(\Omega))} + \| \usf_0\|_{\Hs}^2 \right),
\]
and, therefore, the second term on the right hand side of \eqref{eq:csm} can be bounded by
\begin{align*}
  \left| \int_0^T ( \tr ( \bar{e}(t) - \hat{e}(t) ), \tr \bar{e}(t) )_{L^2(\Omega)} \diff t \right| 
  & \leq \frac14 \| \tr \bar{e} \|_{L^2(0,T;L^2(\Omega))}^2 \\
  &+ C \tau^2 \left( \| \fsf^{\tau}\|^2_{\ell^2(L^2(\Omega))} + \| \usf_0\|_{\Hs}^2 \right).
\end{align*}

Collecting all the derived bounds we arrive at the desired error estimate \eqref{estimate_11}.
\end{proof}

With this estimate at hand we can control the difference  between the fully and the semi-discrete problems.

\begin{theorem}[auxiliary error estimate: $\gamma = 1$]
\label{TH:order_1aux}
Let $\gamma = 1$ and assume that $\usf_0 \in \mathbb{H}^{1+s}(\Omega)$ and $\fsf \in L^2(0,T;\Ws)$.
If $V^\tau$ and $V^{\tau}_{\T_{\Y}}$ solve \eqref{eq:semi_1} and \eqref{fully_beta}, respectively,  then
\begin{equation*}
\| \tr (V^\tau - V^{\tau}_{\T_{\Y}}) \|_{\ell^2(L^2(\Omega))}  \lesssim  |\log N|^{2s} N^{-\frac{1+s}{n+1}} \left( \| \fsf^{\tau} \|_{\ell^2(0,T;\Ws)} + \| \usf_0\|_{\mathbb{H}^{1+s}(\Omega)} \right),
\end{equation*}
where the hidden constant does not depend on $v$, $V^{\tau}$ nor the problem data.
\end{theorem}
\begin{proof}
We start by defining the error
\begin{equation}
\label{En}
E^\tau = (V^\tau - G_{\T_{\Y}} V^\tau) + (G_{\T_{\Y}} V^\tau - V_{\T_{\Y}}^\tau) = \theta^\tau + \rho_{\T_\Y}^\tau,
\end{equation}
where $G_{\T_{\Y}}$ is defined in \eqref{elliptic_projection}. We estimate $\theta^\tau$ by invoking the approximation properties \cite[Proposition 4.7]{NOS3} of $G_{\T_{\Y}}$ and the regularity results of \cite[Theorem 2.7]{NOS3}:
\[
  \| \tr \theta^{\tau} \|_{\ell^2(L^2(\Omega))} \lesssim |\log N|^{2s} N^{-\frac{1+s}{n+1}}
  \left( \| \fsf^\tau \|_{\ell^2(\Ws)} + \| \usf_0 \|_{\mathbb{H}^{1+s}(\Omega)} \right),
\]

The estimate of $\rho^\tau_{\T_\Y}$ follows along the same lines of \cite[Lemma 5.6]{BGRV}. For brevity, we skip the details.
\end{proof}

We collect the estimates of Theorems \ref{TH:order_1} and \ref{TH:order_1aux} to derive a $L^2(Q)$-error estimate.

\begin{theorem}[error estimate for $v$: $\gamma = 1$]
\label{TH:order_1fd}
Assume that $\gamma = 1$ and let $v$ and $V^{\tau}_{\T_{\Y}}$ solve \eqref{state_equation_truncated} and \eqref{fully_beta}, respectively. If $\fsf \in L^{\infty}(0,T;L^2(\Omega)) \cap L^2(0,T;\mathbb{H}^{1-s}(\Omega))$ and $\usf_0 \in \mathbb{H}^{1+s}(\Omega)$, then
\begin{align*}
  \| \tr( v - V^{\tau}_{\T_{\Y}}) \|_{L^2(Q) }  &\lesssim  \tau 
  \left( \| \fsf \|_{L^{\infty}(0,T;L^2(\Omega))} + \| \usf_0 \|_{\mathbb{H}^{s}(\Omega)} \right)
  \\
  &+ |\log N|^{2s} N^{-\frac{1+s}{n+1}}
  \left( \| \fsf \|_{L^2(0,T;\mathbb{H}^{1-s}(\Omega))} + \| \usf_0 \|_{\mathbb{H}^{1+s}(\Omega)} \right),
\end{align*}
where the hidden constant does not depend on $v$, $V^{\tau}$ nor the problem data.
\end{theorem}

\begin{corollary}[error estimate for $\usf$: $\gamma = 1$]
\label{CR:order_1fd}
Assume that $\gamma =1$ and  let $\usf$ solve \eqref{fractional_heat} with $\zsf =0$ and $U^{\tau}_{\T_{\Omega}}$ be defined by \eqref{discrete_Ufd}. If $\fsf \in L^{\infty}(0,T;L^2(\Omega)) \cap L^2(0,T;\mathbb{H}^{1-s}(\Omega))$ and $\usf_0 \in \mathbb{H}^{1+s}(\Omega)$, then
\begin{align*}
  \| \usf - U^{\tau}_{\T_{\Omega}} \|_{L^2(0,T; L^2(\Omega)) } &\lesssim  \tau 
  \left(  \| \fsf \|_{L^{\infty}(0,T;L^2(\Omega))} + \| \usf_0 \|_{\mathbb{H}^{s}(\Omega)} \right)
  \\
  &+ |\log N|^{2s} N^{-\frac{1+s}{n+1}} \left( \| \fsf \|_{L^2(0,T;\mathbb{H}^{1-s}(\Omega))} 
  + \| \usf_0 \|_{\mathbb{H}^{1+s}(\Omega)}\right),
\end{align*}
where the hidden constant does not depend on $v$, $V^{\tau}$ nor the problem data.
\end{corollary}
\begin{proof}
The result is a consequence of Theorem \ref{TH:order_1fd} and \cite[Lemma 4.3]{NOS3}, in conjunction with the appropriate choice of $\Y$ explored in \cite[Remark 5.5]{NOS}.
\end{proof}

\section{Approximation of the fractional control problem}
\label{sec:approximation_control}

We propose an implicit fully-discrete scheme to approximate the solution of the fractional control problem \eqref{Jintro}--\eqref{cc}: piecewise constant functions for the control and, for the state, first degree tensor product finite elements in space, as described in \S\ref{sec:space_discretization}, and the finite difference discretization in time detailed in \S\ref{sec:time_discretization}.

As stated in Theorem~\ref{th:order_beta}, in order to have the error estimates \eqref{estimate_1} and \eqref{estimate_2} for the approximation of the state equation \eqref{fractional_heat}, we have to require that $\mathcal{A}(\usf_0,\fsf + \orsf) < \infty$. This strong $H^2$ in time regularity assumption is not satisfied by the optimal control $\orsf$, meaning that we are not able to apply the results of Theorem~\ref{th:order_beta}. This is in sharp contrast with the case $\gamma = 1$ which, according to Theorem~\ref{TH:order_1fd} only requires $\fsf, \orsf \in L^{\infty}(0,T;L^2(\Omega)) \cap L^{2}(0,T;\mathbb{H}^{1-s}(\Omega))$ which, by imposing \eqref{ab_condition} and invoking Remark~\ref{rm:regularity_control_r} and Theorem \ref{th:regofz}, is satisfied by the optimal control $\orsf$. Due to this regularity restriction we can obtain an error analysis for $\gamma = 1$ only. We remark that $L^2(Q)$-error estimates, for $s,\gamma \in (0,1)$ are not available in the literature, especially under the correct regularity assumptions. 
In \S\ref{sub:apriori_control} we will present error estimates for $\gamma = 1$ and $s \in (0,1)$, and
in \S\ref{sub:convergence} we will show the convergence, without rates, for the remaining range of parameters.

Finally, to simplify the exposition, in what follows we assume that $\asf$ and $\bsf$ are constants that satisfy \eqref{ab_condition}.

\subsection{An implicit fully discrete-scheme}
\label{sub:fd_control}
To discretize the control we introduce the finite element space of piecewise constant functions over $\T_\Omega$ 
\[
  \mathbb{Z}(\T_\Omega) = \left\{ Z \in L^\infty(\Omega): Z|_K \in \mathbb{P}_0(K), \quad \forall K \in \T_\Omega \right\},
\]
and the space of piecewise constant functions in time and space
\begin{equation}
\mathbb{Z}(\mathcal{T},\T_\Omega) = \left\{ Z^\tau \subset L^\infty(Q) : Z^k \in \mathbb{Z}(\T_\Omega) \right\}.
\end{equation}
We define the space of discrete admissible controls as follows: 
\begin{equation}
\label{eq:defofZadh}
  \Zad(\mathcal{T},\T_\Omega) = \Zad \cap \mathbb{Z}(\mathcal{T},\T_\Omega).
\end{equation}

It will be useful to introduce the $L^2(Q)$-orthogonal projection onto $\mathbb{Z}(\calT,\T_\Omega)$. The operator $\Pi_{\T_{\Omega}}^{\mathcal{T}}: L^2(Q) \rightarrow \mathbb{Z}(\mathcal{T},\T_\Omega)$ is defined by
\begin{equation}
\label{eq:or_pro}
  r \in L^2(Q): \quad  (r -  \Pi^{\mathcal{T}}_{\T_{\Omega}}  r, Z)_{L^2(Q)} = 0 \qquad \forall Z \in \mathbb{Z}(\mathcal{T},\T_\Omega),
\end{equation}
and, for all $r \in H^1(0,T;L^2(\Omega)) \cap L^2(0,T;H^1(\Omega))$, satisfies:
\begin{equation}
 \label{eq:or_pro_prop}
 \| r - \Pi^{\mathcal{T}}_{\T_{\Omega}}r \|_{L^2(Q)} \lesssim h_{\T_{\Omega}} \| \nabla_{x'} r \|_{L^2(Q)} + \tau \| \partial_t r \|_{L^2(Q)}.
\end{equation}
Notice also that, since $\asf$ and $\bsf$ are constant, $\Pi_{\T_\Omega}^\calT \Zad \subset \Zad(\calT, \T_\Omega)$.

We define the discrete functional
\[
J_{\T_{\Y}}^{\calT}(V^\tau_{\T_\Y},Z_{\T_\Omega}) = \frac{1}{2} \| \tr V_{\T_{\Y}}^{\tau} - \usf_d^{\tau}\|^2_{\ell^2(L^2(\Omega))} +  \frac{\mu}{2}\| Z^{\tau}_{\T_{\Omega}} \|^2_{\ell^2(L^2(\Omega))}
\]
where the $\ell^2$-norm is defined in \S\ref{sec:time_discretization}. The identification between a sequence $\phi^{\tau}$ and the piecewise constant function \eqref{eq:equivalence} will be used repeatedly  below. For instance, if $\usf_d^\tau = \usf_d$ we would have that $J_{\T_\Y}^{\calT}(w,r) = J(w,r)$ whenever $w^\tau = w$ and $r^\tau = r$, that is, the arguments are piecewise constant over $\calT$. This was already implicitly used in \eqref{eq:defofZadh}, when we defined $\Zad(\mathcal{T},\T_\Omega)$.

The numerical scheme reads: Find
$
 \min J_{\T_\Y}^{\calT}(V_{\T_{\Y}^\tau},Z_{\T_{\Omega}}^\tau), 
$
subject to the discrete state equation: initialize as in \eqref{initial_data_discrete} and for $k=0,\dots,\K-1$, let $V_{\T_{\Y}}^{k+1} \in \V(\T_{\Y})$ solve
\begin{equation}
\label{fully_state}
 ( \delta^\gamma \tr V_{\T_{\Y}}^{k+1} , \tr W )_{L^2(\Omega)}  + a_\Y(V_{\T_{\Y}}^{k+1},W) = \left\langle \fsf^{k+1} + Z_{\T_{\Omega}}^{k+1}, \tr W   \right\rangle,
\end{equation}
for all $W \in \V(\T_\Y)$ and the control constraints $Z^{\tau}_{\T_\Omega} \subset \Zad(\calT, \T_{\Omega})$. If $(\bar{V}_{\T_\Y}^\tau, \bar{Z}_{\T_\Omega}^\tau)$ denote the solution to this problem, setting
\begin{equation}
\label{eq:U_fd}
\bar{U}^{\tau}_{\T_\Omega} = \tr \bar{V}_{\T_{\Y}}^{\tau}
\end{equation}
we obtain a fully-discrete approximation $(\bar{U}^{\tau}_{\T_\Omega},\bar{Z}^{\tau}_{\T_{\Omega}}) \in \U(\T_{\Omega})^{\K}\times\Zad(\calT,\T_{\Omega})$ to the fractional control problem \eqref{Jintro}--\eqref{cc}.

\begin{remark}[locality]\rm
The main advantage of the scheme \eqref{fully_state} approximating the fractional control problem \eqref{Jintro}--\eqref{cc} via \eqref{eq:U_fd} is its \emph{local nature}.
\end{remark}

\subsection{A priori error analysis: $\gamma = 1$ and $s \in (0,1)$}
\label{sub:apriori_control}
Let us consider $s \in (0,1)$ and $\gamma = 1$ in \eqref{state_equation_extended}--\eqref{extended_adjoint} and provide an a priori error analysis for the fully-discrete scheme proposed in \S\ref{sub:fd_control}. To do so, we provide first order necessary and sufficient optimality conditions of the fully-discrete problem. We define the discrete adjoint problem: Find $P_{\T_{\Y}}^{\tau} \subset \V(\T_{\Y})$ such that $\tr P_{\T_{\Y}}^{\K} = 0$, and for $k = \K-1, \ldots, 0$, $P_{\T_{\Y}}^k \in \V(\T_\Y)$ solves
\begin{equation}
\label{fully_adjoint}
 ( \bar{\delta}^1 \tr P_{\T_{\Y}}^{k} , \tr W )_{L^2(\Omega)}  + a_\Y(P_{\T_{\Y}}^k,W) = \left\langle \tr V_{\T_{\Y}}^{k+1} - \usf_d^{k+1}, \tr W   \right\rangle,
\end{equation} 
for all $W \in \V(\T_{\Y})$. Here $\bar{\delta}^1$ denotes
$
  \bar{\delta}^1 \phi^k = -\tau^{-1} \left( \phi^{k+1} - \phi^k \right).
$
The optimality condition reads: $( \bar{V}^{\tau}_{\T_\Y}, \bar{Z}^{\tau}_{\T_{\Omega}})$ is optimal if and only if $\bar{V}^{\tau}_{\T_\Y}$ solves
\eqref{initial_data_discrete} and \eqref{fully_state} and 
\begin{equation}
\label{eq:op_discrete}
 ( \tr \bar{P}_{\T_{\Y}}^{\tau} + \mu \bar{Z}^{\tau}_{\T_{\Omega}}, Z -\bar{Z}^{\tau}_{\T_{\Omega}})_{L^2(Q)} \geq 0 \qquad 
 \forall Z \in \Zad(\mathcal{T},\T_\Omega),
\end{equation}
where $\bar{P}^{\tau}_{\T_{\Y}}$ solves \eqref{fully_adjoint}. Notice that \eqref{eq:op_discrete} can be equivalently written as
\[
  ( \tr \bar P_{\T_\Y}^k + \mu \bar{Z}^k_{\T_\Omega}, Z - \bar{Z}^k_{\T_\Omega} )_{L^2(\Omega)} \quad \forall Z \in \mathbb{Z}(\T_\Omega), \quad \asf \leq Z \leq \bsf,
  \quad \forall k = 1,\ldots,\K.
\]
To see this, it suffices to set $Z^\tau = Z \chi_{(t_{k-1},t_k]}$, with $Z \in \mathbb{Z}(\T_\Omega)$ and $\asf \leq Z \leq \bsf$. This greatly simplifies the implementation.

Let us now introduce two auxiliary problems. The first one reads: Find $Q_{\T_{\Y}}^{\tau} \subset \V(\T_{\Y})$ such that $\tr Q_{\T_{\Y}}^{\K} = 0$ and, for $k = \K-1, \ldots, 0$, $Q_{\T_{\Y}}^k \in \V(\T_\Y)$ solves
\begin{equation}
\label{adjoint_aux1}
 ( \bar{\delta}^1 \tr Q_{\T_{\Y}}^{k} , \tr W )_{L^2(\Omega)}  + a_\Y(Q_{\T_{\Y}}^{k},W) = \left\langle \tr \bar{v}^{k+1} - \usf_d^{k+1}, \tr W   \right\rangle, 
\end{equation}
for all $W \in \V(\T_\Y)$, and where $\bar{v} = \bar{v}(\orsf)$ solves \eqref{state_equation_truncated}.
The second one is: Find $R_{\T_{\Y}}^{\tau} \subset \V(\T_{\Y})$ such that $\tr R_{\T_{\Y}}^{\K} = 0$, and for $k = \K-1, \ldots, 0$, $R_{\T_{\Y}}^k \in \V(\T_\Y)$ solves
\begin{equation}
\label{adjoint_aux2}
 ( \bar{\delta}^1 \tr R_{\T_{\Y}}^{k} , \tr W )_{L^2(\Omega)}  + a_\Y(R_{\T_{\Y}}^{k},W) = \left\langle \tr V^{k+1}_{\T_{\Y}}(\orsf) - \usf_d^{k+1}, \tr W   \right\rangle,
\end{equation}
for all $W \in \V(\T_\Y)$. These auxiliary problems will allow us to derive error estimates for the fully-discrete scheme proposed in \S \ref{sub:fd_control}.

\begin{lemma}[error estimate for the control: $\gamma = 1$ and $s \in (0,1)$]
\label{LM:control_error}
Let $\bar{\rsf}$ be the solution to the truncated optimal control problem of \S\ref{sec:control_truncated} and let $\bar{Z}^{\tau}_{\T_{\Omega}}$ be the solution to the fully-discrete optimal control problem of \S\ref{sub:fd_control}. Assume that $\usf_0 \in \mathbb{H}^{1+s}(\Omega)$ and, for every $\epsilon>0$, $\usf_d \in L^2(0,T;\mathbb{H}^{1-\epsilon}(\Omega)) \cap H^1(0,T;L^2(\Omega))$ and $\fsf \in H^1(0,T;\mathbb{H}^{1-\epsilon}(\Omega))$. Then
\begin{align*}
\| \bar{\rsf} - \bar{Z}^{\tau}_{\T_{\Omega}}  \|_{L^2(Q)} \lesssim \tau + |\log N|^{2s}N^{-\frac{1}{n+1}},
\end{align*}
where the hidden constant is independent of the discretization parameters but depends on the problem data.
\end{lemma}
\begin{proof} 
We proceed in several steps.

\noindent \framebox{1} Setting $\rsf = \bar{Z}^{\tau}_{\T_\Omega}$ and $Z = \Pi^{\mathcal{T}}_{\T_{\Omega}} \orsf$ in \eqref{op_truncated} and \eqref{eq:op_discrete}, respectively, and adding the derived inequalities we arrive at
\[
 \mu \| \bar{\rsf} - \bar{Z}^{\tau}_{\T_{\Omega}}  \|_{L^2(Q)}^2 \leq (\tr(\bar{p} - \bar{P}^{\tau}_{\T_{\Y}}),\bar{Z}^{\tau}_{\T_\Omega} - \orsf )_{L^2(Q)} + (\tr \bar{P}^{\tau}_{\T_{\Y}} + \mu \bar{Z}^{\tau}_{\T_\Omega},\Pi^{\mathcal{T}}_{\T_{\Omega}} \orsf - \orsf )_{L^2(Q)},
\]
where $\Pi^{\mathcal{T}}_{\T_{\Omega}}$ is defined in \eqref{eq:or_pro}. 

\noindent \framebox{2} Using the solution to \eqref{adjoint_aux1} we write $\bar{p} - \bar{P}^{\tau}_{\T_{\Y}} = (\bar{p} - Q^{\tau}_{\T_{\Y}}) + (Q^{\tau}_{\T_{\Y}} - \bar{P}^{\tau}_{\T_{\Y}})$. The first term is estimated by using the results of Theorem \ref{TH:order_1fd} as follows:
\begin{multline*}
\| \tr(\bar{p} - Q^{\tau}_{\T_{\Y}}) \|_{L^2(Q)} \lesssim
\tau \left( \| \tr \bar{v}  \|_{L^{\infty}(0,T;L^2(\Omega))} + \| \usf_d  \|_{L^{\infty}(0,T;L^2(\Omega))} 
+ \| \usf_0 \|_{\mathbb{H}^{s}(\Omega)} \right)
\\
+ |\log N|^{2s} N^{-\frac{1+s}{n+1}}
\left( \| \tr \bar{v}\|_{L^2(0,T;\mathbb{H}^{1-s}(\Omega))} + \| \usf_d \|_{L^2(0,T;\mathbb{H}^{1-s}(\Omega))}+ 
\| \usf_0 \|_{\mathbb{H}^{1+s}(\Omega)} \right).
\end{multline*}
Where we used that $\tr \bar v \in L^{\infty}(0,T;L^2(\Omega)) \cap L^2(0,T;\Ws)$, which follows from Remark~\ref{rk:reg_states} and Remark~\ref{rm:regularity_control_r}.

\noindent \framebox{3} To estimate the difference $\tr(Q^{\tau}_{\T_{\Y}} - \bar{P}^{\tau}_{\T_{\Y}})$,
we write $Q^{\tau}_{\T_{\Y}} - \bar{P}^{\tau}_{\T_{\Y}} =  (Q^{\tau}_{\T_{\Y}} - R^{\tau}_{\T_{\Y}})
+ (R^{\tau}_{\T_{\Y}} - \bar{P}^{\tau}_{\T_{\Y}})$, where $R^{\tau}_{\T_{\Y}}$ solves \eqref{adjoint_aux2}.
Employing the stability estimate established in \cite[Lemma 5.1]{NOS3} and Theorem \ref{TH:order_1fd}, we arrive at
\begin{align*}
\| \tr(Q^{\tau}_{\T_{\Y}} - R^{\tau}_{\T_{\Y}})\|_{L^2(Q)} & \lesssim 
\| \tr( \bar{v} - V^{k+1}_{\T_{\Y}}(\orsf) )\|_{L^2(Q)}
\\
& \lesssim \tau \left( \| \orsf \|_{L^{\infty}(0,T;L^2(\Omega))} + \| \usf_0 \|_{\mathbb{H}^{s}(\Omega)} \right)
\\
& + |\log N|^{2s} N^{-\frac{1+s}{n+1}} \left( \| \orsf \|_{L^2(0,T;\mathbb{H}^{1-s}(\Omega))} + \| \usf_0 \|_{\mathbb{H}^{1+s}(\Omega)} \right).
\end{align*}
To handle the term $R^{\tau}_{\T_{\Y}} - \bar{P}^{\tau}_{\T_{\Y}}$ we invoke the discrete counterpart of Step 2 in Lemma~\ref{LE:exp_convergence}, that is an argument based on summation by parts, to arrive at
\[
  ( \tr (R^{\tau}_{\T_{\Y}} - \bar{P}^{\tau}_{\T_{\Y}}), \bar{Z}^{\tau}_{\T_\Omega} - \orsf^\tau )_{L^2(Q)} \leq 0.
\]
\noindent \framebox{4} Using the solutions to \eqref{adjoint_aux1} and \eqref{adjoint_aux2} we write 
\begin{multline*}
  (\tr \bar{P}^{\tau}_{\T_{\Y}} + \mu \bar{Z}^{\tau}_{\T_\Omega}, \Pi^{\mathcal{T}}_{\T_{\Omega}} \orsf - \orsf )_{L^2(Q)} = 
  ( \tr \bar p + \mu \orsf, \Pi_{\T_\Omega}^\calT \orsf - \orsf )_{L^2(Q)} \\
  + ( \tr (\bar{P}_{\T_\Y}^\tau \pm Q_{\T_\Y}^\tau - \bar p), \Pi_{\T_\Omega}^\calT \orsf - \orsf )_{L^2(Q)}
  + \mu (\bar{Z}_{\T_\Omega}^\tau - \orsf, \Pi_{\T_\Omega}^\calT \orsf - \orsf )_{L^2(Q)}
  = \textrm{I} + \textrm{II} + \textrm{III}.
\end{multline*}
Using the properties of the projector $\Pi_{\T_\Omega}^\calT$ and the smoothness of $\bar p$ and $\orsf$ we have
\begin{align*}
\textrm{I} & = ( \tr \bar{p} + \mu \orsf - \Pi^{\mathcal{T}}_{\T_{\Omega}}( \tr \bar{p} + \mu \orsf ), \Pi^{\mathcal{T}}_{\T_{\Omega}} \orsf - \orsf )
\lesssim \big( \tau \|\tr \bar{p} + \mu \orsf \|_{H^1(0,T;L^2(\Omega))} 
\\
& + h_{\T_{\Omega}}\|\tr \bar{p} + \mu \orsf \|_{L^2(0,T;H^1(\Omega))} \big) 
( \tau \| \orsf \|_{H^1(0,T;L^2(\Omega))}  + h_{\T_{\Omega}} \| \orsf \|_{L^2(0,T;H^1(\Omega))}).
\end{align*}
The term $\textrm{II}$ can be handled by repeating the arguments of Steps 2 and 3, while $\textrm{III}$ is controlled by a trivial aplicaiton of the Cauchy Schwarz inequality.


\noindent \framebox{5} The assertion follows from collecting all the estimates we obtained in previous steps and recalling that $h_{\T_\Omega} \approx N^{-\frac1{n+1}}$.
\end{proof}


On the basis of of Lemma~\ref{LM:control_error} we derive the following important result.

\begin{theorem}[control error estimates: $s \in (0,1)$ and $\gamma = 1$]
\label{thm:f_c_est}
Let $\bar{\zsf}$ be the solution to the space-time fractional optimal control problem \eqref{Jintro}--\eqref{cc} and let $\bar{Z}^{\tau}_{\T_{\Omega}}$ be the solution to the fully-discrete optimal control problem of \S\ref{sub:fd_control}. In the framework of Lemma \ref{LM:control_error}, we have the following error estimate
\begin{equation*}
\| \ozsf - \bar{Z}^{\tau}_{\T_{\Omega}} \|_{L^2(Q)} \lesssim \tau + |\log N|^{2s} N^{-\frac{1}{n+1}},
\end{equation*}
where the hidden constant is independent of the discretization parameters but depends on the problem data.
\end{theorem}
\begin{proof} 
The result follows from Lemmas \ref{LE:exp_convergence} and \ref{LM:control_error}, in conjunction with an appropriate selection of the parameter $\Y$, as detailed in \cite[Remark 5.5]{NOS} and \cite[Corollary 5.17]{AO}.
\end{proof}

We conclude with an error estimate for the state in the $L^2(0,T;\Hs)$-norm.

\begin{theorem}[state error estimates: $s \in (0,1)$ and $\gamma = 1$]
\label{thm:stateerror}
Let $\bar{\usf}$ be the optimal state of the space-time fractional optimal control problem \eqref{Jintro}--\eqref{cc} and let $\bar{U}^{\tau}_{\T_{\Omega}}$ be defined as in \eqref{eq:U_fd}. In the framework of Lemma \ref{LM:control_error}, we have the following error estimate
\[
  \| \ousf - \bar{U}^\tau_{\T_\Omega} \|_{L^2(0,T;\Hs)} \lesssim \tau + |\log N|^{2s} N^{\frac{-1}{n+1}}
\]
where the hidden constant is independent of the discretization parameters but depends on the problem data.
\end{theorem}
\begin{proof}
We first write
\[
\| \ousf - \bar{U}_{\T_\Omega}^\tau \|_{L^2(0,T;\Hs)} \leq \| \tr(\bar{\ue} - \bar v) \|_{L^2(0,T;\Hs)}
  + \| \tr \bar{v} - \bar{U}^\tau_{\T_\Omega} \|_{L^2(0,T;\Hs)},
\]
and note that the first term is controlled in \eqref{state_exp}. The second term is handled by noticing that
\begin{multline*}
  \| \tr\bar{v} - \bar{U}_{\T_\Omega}^\tau  \|_{L^2(0,T;\Hs)} =
  \| \tr (\bar{v} - \bar{V}^\tau_{\T_\Y} ) \|_{L^2(0,T;\Hs)} \\
  \leq
  \| \tr( \bar{v} - V^\tau_{\T_\Y}(\orsf) )  \|_{L^2(0,T;\Hs)}
  + \| \tr( V^\tau_{\T_\Y}(\orsf) - \bar{V}^\tau_{\T_\Y} ) \|_{L^2(0,T;\Hs)}.
\end{multline*}
Since $\fsf + \orsf \in H^1(0,T;L^2(\Omega))$, the first term on the right hand side of this inequality is estimated using the error estimates for the discrete scheme presented in \cite[Theorem 5.4]{NOS3}.
The second one can be handled by invoking the stability of the discrete scheme and the error estimates of Theorem~\ref{thm:f_c_est}. Collecting these bounds we obtain the result.
\end{proof}

\subsection{Convergence}
\label{sub:convergence}
Let us now consider the case when either $\gamma,s \in (0,1)$, or the problem data is not smooth enough to yield the error estimates of \S\ref{sub:apriori_control} and elucidate the general convergence properties of the fully discrete scheme. Notice that we are not only approximating the state equation via discretization, but we are also approximating the cost, so convergence of discrete optimal controls to the continuous one is not immediate. To begin, as in Definition~\ref{def:fractional_operator}, we introduce the discrete control to state operator
\begin{equation}
  \mathbf{S}_{\T_\Y}^\calT : \mathbb{Z}(\calT,\T_\Omega) \ni Z_{\T_\Omega}^\tau \mapsto V_{\T_\Y}^\tau \subset \V(\T_\Y),
\end{equation}
where $V^\tau_{\T_\Y}$ solves the discrete state equation \eqref{fully_state}. Notice that the stability estimates implicit in Theorems~\ref{th:order_beta} and \ref{TH:order_1} (\cite[Lemma 5.1]{NOS3}) yield that, for all $\gamma,s \in (0,1)$ the family $\{\mathbf{S}_{\T_\Y}^\calT\}$ is uniformly bounded for  $\fsf,Z_{\T_\Omega}^\tau \in L^2(Q)$. Moreover, the error estimates imply the pointwise convergence of these operators so that, by the uniform boundedness principle, they converge uniformly to $\mathbf{S}$. This will be crucial in showing convergence.

With the discrete control to state operators at hand, like in \eqref{f_opt}, we define the reduced cost functional by
\begin{equation}
\label{eq:reduced_cost}
  \begin{aligned}
    F_{\T_\Y}^{\calT}( Z_{\T_\Omega}^\tau) &=
    J_{\T_\Y}^{\calT}( \mathbf{S}_{\T_\Y}^\calT Z_{\T_\Omega}^\tau, Z_{\T_\Omega}^\tau ) \\
    &= \frac12 \| \tr \mathbf{S}_{\T_\Y}^\calT Z_{\T_\Omega}^\tau - \usf_d^\tau \|_{\ell^2(L^2(\Omega))}^2
    + \frac\mu2 \| Z_{\T_\Omega}^\tau \|_{\ell^2(L^2(\Omega))}^2.
  \end{aligned}
\end{equation}

The convergence of the fully discrete scheme is the content of the next result.

\begin{theorem}[convergence]
\label{th:gamma}
The family $\{ \bar{Z}_{\T_\Omega}^\tau \}_{\T_\Omega \in \Tr_\Omega, \tau>0}$ is uniformly bounded and it contains a subsequence that converges $L^2(Q)$-weak to $\orsf$, the solution to the truncated optimal control problem. Moreover, if $\gamma = 1$ the convergence is strong.
\end{theorem}
\begin{proof}
Boundedness follows immediately from the fact that $\bar{Z}_{\T_\Omega}^\tau$ minimizes $F_{\T_\Y}^{\calT}$. If $z_0 \in \Zad$, then
\[
  F_{\T_\Y}^{\calT}(\bar{Z}_{\T_\Omega}^\tau) \leq F_{\T_\Y}^{\calT} (\Pi_{\T_\Omega}^\calT z_0 ) \lesssim \| z_0 \|_{L^2(Q)}^2 + \| \usf_d \|_{L^2(Q)}^2,
\]
where we used the uniform boundedness of $\Pi_{\T_\Omega}^\calT$ and $\mathbf{S}_{\T_\Y}^\calT$. This implies the existence of a (not relabeled) weakly convergent subsequence.

To show convergence of this subsequence to $\orsf$ we appeal to the theory of $\Gamma$-convergence, for which we need to verify several assumptions:

\noindent \framebox{1} \emph{Lower bound inequality}: Assume that $Z_{\T_\Omega}^\tau \rightharpoonup z$ in $L^2(Q)$. For $w \in L^2(Q)$, we have
\[
  (  \mathbf{S}_{\T_\Y}^\calT Z_{\T_\Omega}^\tau - \mathbf{S}z , w )_{L^2(Q)} = 
  (  \mathbf{S}_{\T_\Y}^\calT z - \mathbf{S}z , w )_{L^2(Q)} 
  + ( \mathbf{S}_{\T_\Y}^\calT ( Z_{\T_\Omega}^\tau - z ), w )_{L^2(Q)}
  = \textrm{I} + \textrm{II}.
\]
The pointwise convergence of $\mathbf{S}_{\T_\Y}^\calT$ to $\mathbf{S}$ shows that $\textrm{I} \to 0$, while their uniform convergence that $\textrm{II} \to 0$. In conclusion $\mathbf{S}_{\T_\Y}^\calT Z_{\T_\Omega}^\tau \rightharpoonup \mathbf{S} z$. Lower semicontinuity of the norms and $\usf_d^\tau \to \usf_d$ in $L^2(Q)$ imply
\[
  f(z) \leq \liminf F_{\T_\Y}^{\calT}(Z_{\T_\Omega}^\tau),
\]
which is what we needed to show.

\noindent \framebox{2} \emph{Existence of a recovery sequence}: Let $z \in \Zad$, then $\Pi_{\T_\Omega}^\calT z \in \Zad(\calT, \T_\Omega)$ converges strongly to $z$ in $L^2(Q)$. Consequently, $\mathbf{S}_{\T_\Y}^\calT \Pi_{\T_\Omega}^\calT z \to \mathbf{S}z$ in $L^2(Q)$ as well. The continuity of $F_{\T_\Y}^{\calT}$ then implies
\[
  f(z) \geq \limsup F_{\T_\Y}^{\calT}(\Pi_{\T_\Omega}^\calT z).
\]

\noindent \framebox{3} \emph{Equicoerciveness}: Since
\[
  F_{\T_\Y}^{\calT} (Z) \geq \frac\mu2 \| Z \|_{L^2(Q)}^2,
\]
we have, by \cite[Proposition 7.7]{MR1201152}, that the family $\{F_{\T_\Y}^{\calT}\}$ is equicoercive.

\noindent \framebox{4} Steps 1 and 2 show the $\Gamma$-convergence of the discrete reduced costs $F_{\T_\Y}^{\calT}$ to the reduced cost $f$. This implies, using \cite[Corollary 7.20]{MR1201152}, that minimizers of $F_{\T_\Y}^{\calT}$, if they converge, must do so to a minimizer of $f$. Step 3 and the uniqueness of the minimizer of the reduced cost $f$ are the conditions for the fundamental lemma of $\Gamma$-convergence \cite[Corollary 7.24]{MR1201152}. In conclusion, $\{\bar{Z}_{\T_\Omega}^\tau\}$ converges weakly to $\orsf$, the minimum of the truncated cost functional.

We conclude with the strong convergence for the case of $\gamma =1$, which follows from the a priori estimates. Namely, a basic energy estimate for scheme \eqref{fully_state}, together with a slight modification of Lemma~\ref{LM:stab_sd} imply that
\[
  \tau^{-1} \| \tr (\mathbf{S}_{\T_\Y}^\tau \bar{Z}_{\T_\Omega}^\tau 
    - \mathpzc{S}_\tau \mathbf{S}_{\T_\Y}^\tau \bar{Z}_{\T_\Omega}^\tau) \|_{L^\infty(\tau,T;L^2(\Omega))}
  + \| \tr \mathbf{S}_{\T_\Y}^\tau \bar{Z}_{\T_\Omega}^\tau \|_{L^2(0,T;\Hs)} \lesssim 1,
\]
where $\mathpzc{S}_\tau$ indicates the backward shift (in time) operator. These are, according to \cite[Theorem 1]{MR2890969}, sufficient conditions for relative compactness of the family 
\[
  \{\mathbf{S}_{\T_\Y}^\tau \bar{Z}_{\T_\Omega}^\tau\}_{\T_\Omega \in \Tr_\Omega, \tau>0}.
\]
By passing to a subsequence we get strong convergence. Denote now by $\bar{P}^{\tau}_{\T_\Omega}$ the solution to the discrete adjoint equations \eqref{fully_adjoint} with right hand side $\mathbf{S}_{\T_\Y}^\tau \bar{Z}_{\T_\Omega}^\tau - \usf_d^\tau$. Stability of \eqref{fully_adjoint} implies that the sequence $\{\bar{P}^{\tau}_{\T_\Omega}\}$ converges strongly in $L^2(Q)$.

Set $\rsf = \bar{Z}^{\tau}_{\T_\Omega}$ and $Z = \Pi^{\mathcal{T}}_{\T_{\Omega}} \orsf$ in \eqref{op_truncated} and \eqref{eq:op_discrete}, respectively. Adding the derived inequalities we arrive at
\[
 \mu \| \bar{\rsf} - \bar{Z}^{\tau}_{\T_{\Omega}}  \|_{L^2(Q)}^2 \leq (\tr(\bar{p} - \bar{P}^{\tau}_{\T_\Omega}),\bar{Z}^{\tau}_{\T_\Omega} - \orsf )_{L^2(Q)}
  + (\tr \bar{P}^{\tau}_{\T_\Omega} + \mu \bar{Z}^{\tau}_{\T_\Omega},\Pi^{\mathcal{T}}_{\T_{\Omega}} \orsf - \orsf )_{L^2(Q)} 
\]
where $\bar{p}$ and $\bar{P}^{\tau}_{\T_\Omega}$ solve \eqref{truncated_adjoint} and \eqref{fully_adjoint}, respectively. Using the definition of $\Pi^\mathcal{T}_{\T_\Omega}$ we immediately have that $(\tr \bar{P}^{\tau}_{\T_\Omega} + \mu \bar{Z}^{\tau}_{\T_\Omega},\Pi^{\mathcal{T}}_{\T_{\Omega}} \orsf - \orsf )_{L^2(Q)}  \to 0$. We handle the remainding term using the strong convergence of $\bar{P}^{\tau}_{\T_\Omega}$.

This concludes the proof.
\end{proof}

\section{Numerical experiments}\label{s:numerics}
Let us illustrate the performance of the fully discrete scheme proposed in \S\ref{sub:fd_control} for $\gamma=1$ and the error estimates derived in \S\ref{sub:apriori_control}.

\subsection{Implementation}
The implementation has been carried out in MATLAB$^\copyright$. 
The stiffness and mass matrices of the discrete system \eqref{fully_state} are assembled exactly, and
the respective forcing boundary term are computed by a quadrature formula which is exact for polynomials of degree $4$. 
The resulting linear system is solved by using the built-in \emph{direct solver} of MATLAB$^\copyright$.
To solve the minimization problem, we use the projected BFGS method with Armijo line search; see \cite{MR1678201}. The optimization algorithm is terminated when the $\ell^2$-norm of the projected gradient is less or equal to $10^{-9}$.

To illustrate the error estimates of \S\ref{sub:apriori_control} we need an exact solution to the fractional control problem \eqref{Jintro}--\eqref{cc}. Let $n = 2$, $\mu =1$, $\Omega = (0,1)^2$, and $\calL = -\Delta$. In this setting, the eigenpairs of $\mathcal{L}$ are:
\[
  \lambda_{k,l} = \pi^2 (k^2 + l^2), \quad 
  \varphi_{k,l}(x'_1,x'_2) = \sin(k \pi x'_1) \sin(l\pi x'_2)  
  \qquad k, l \in \mathbb{N}.
\]
Set $\ousf = e^t \sin(2 \pi x'_1) \sin(2\pi x'_2)$, which yields $\fsf = (1+\lambda_{2,2}^s) e^t \sin(2 \pi x'_1) \sin(2\pi x'_2) - \ozsf$. Set also $\opsf = -\mu (T-t) e^t \sin(2 \pi x'_1) \sin(2 \pi x'_2)$. Definition~\ref{def:fractional_adjoint} then yields
$\usf_d = \big[ 1 - \mu \{ -1 + (1 - \lambda_{2,2}^s)(T-t) \} \big] e^t \sin(2 \pi x'_1) \sin(2\pi x'_2)$.
Finally, we set $\asf = 0$ and $\bsf = 0.5$. The projection formula \eqref{eq:projformula} gives the value of $\ozsf$. This defines, for any $s\in (0,1)$, the data and solution to the optimal control problem \eqref{Jintro}--\eqref{cc}.

\subsection{Convergence rates in space} 

\begin{figure}
\centering
\includegraphics[width=0.44\textwidth]{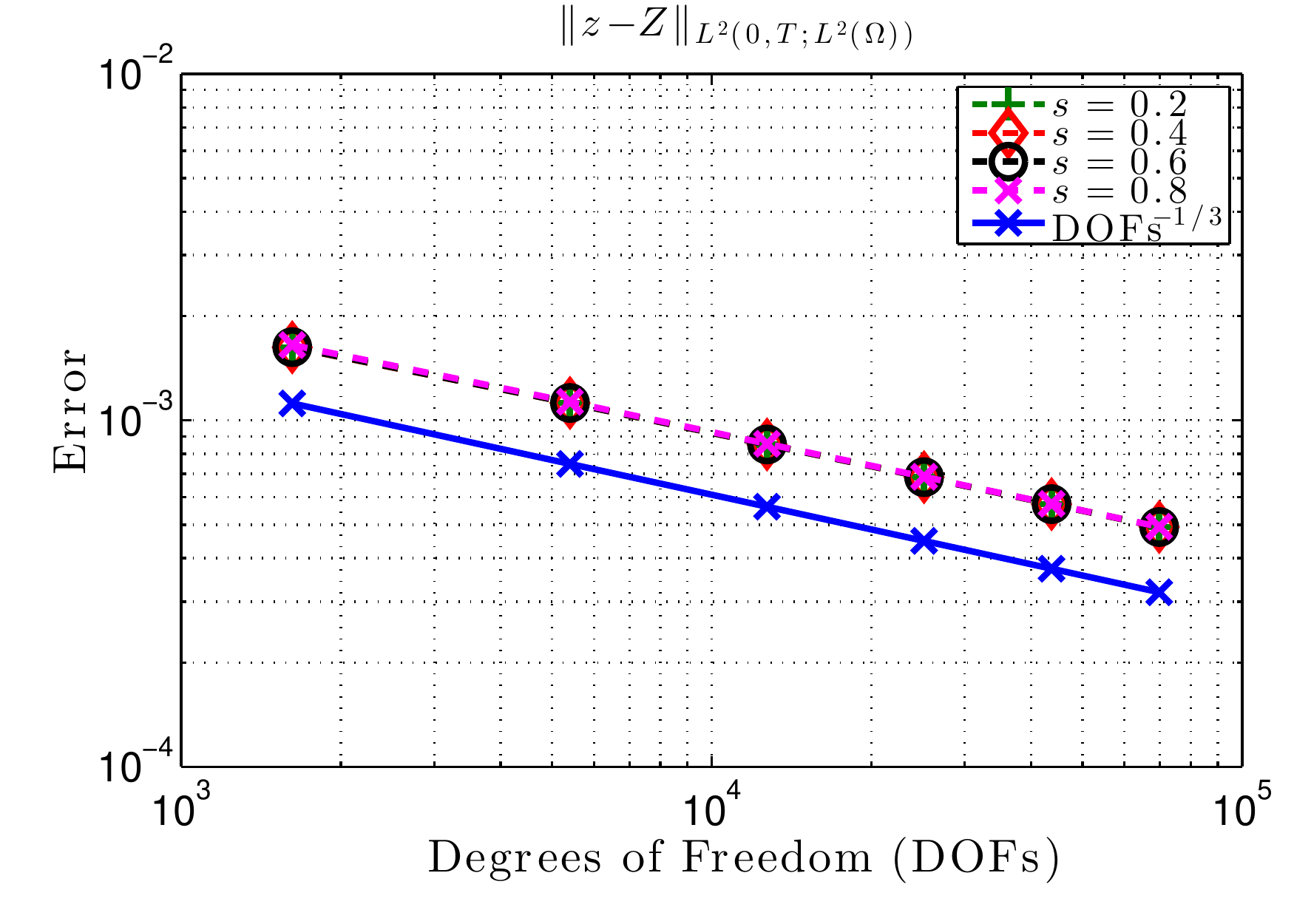} 
\includegraphics[width=0.44\textwidth]{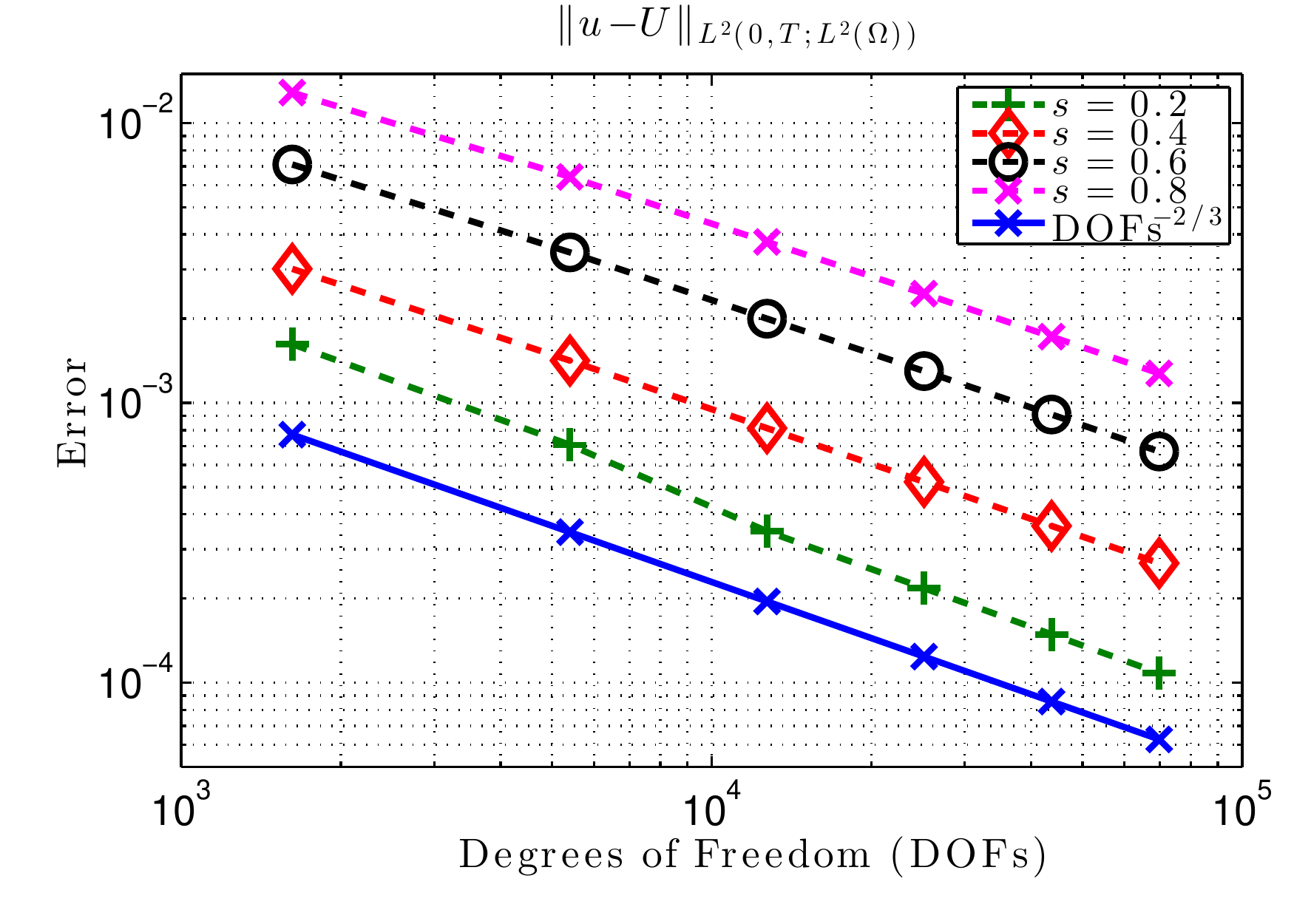}
\caption{\label{f:grad_rate}
Computational rates of convergence for the control and state on anisotropic meshes for
$n = 2$ and $s = 0.2$, $0.4$, $0.6$ and $s = 0.8$. For a fixed number of time steps, $\K = 1400$, 
the left panel shows the decrease of the $L^2(0,T;L^2(\Omega))$-control error with respect to 
$N$ and the right one that of the $L^2(0,T;L^2(\Omega))$-state error. In the control  case we recover 
the rate $N^{-1/3}$. For the state we observe a rate of $N^{-2/3}$.}
\end{figure}

Let $\K = 1400$. The asymptotic relations 
$$\|\ozsf - \bar{Z}_{\T_\Omega}^\tau \|_{L^2(0,T;L^2(\Omega))} \approx N^{-\frac{1}{3}}, 
\qquad
\|\ousf - \bar{U}_{\T_\Omega}^\tau \|_{L^2(0,T;L^2(\Omega))} \approx N^{-\frac{2}{3}} ,
$$ 
are shown in Figure~\ref{f:grad_rate}. The left panel illustrates the quasi-optimal rate of convergence for the optimal control with respect to the number of degrees of freedom $N$ for all choices of the parameter $s$ considered. As noted in \cite{AO,NOS3}, in order to recover optimality, the state and adjoint equations must be discretized with the anisotropic refinement, in the extended dimension, dictated by \eqref{graded_mesh}. 

From Figure~\ref{f:grad_rate} we can also observe that the approximate optimal state converges with a rate $N^{-\frac{2}{3}}$. This rate is not discussed in this paper and will be part of a future work. The theoretical rate of convergence for the approximation of the optimal state is dictated by the results Corollary~\ref{CR:order_1fd}:  $N^{-\frac{1+s}{3}}$, which, in fact, is a consequence of the error estimate derived in \cite[Proposition 4.7]{NOS3}.

\subsection{Convergence rates in time}
Let $N = 927828$. The asymptotic relation 
$$\|\ozsf - \bar{Z}_{\T_\Omega}^\tau \|_{L^2(0,T;L^2(\Omega))} \approx \K^{-1}, 
$$
is shown in Figure~\ref{f:time_refine_rate} and illustrates the optimal decay rate 
in the control with respect to $\K$, for all choices of the 
parameter $s$ considered. 
\begin{figure}
\centering
\includegraphics[width=0.44\textwidth]{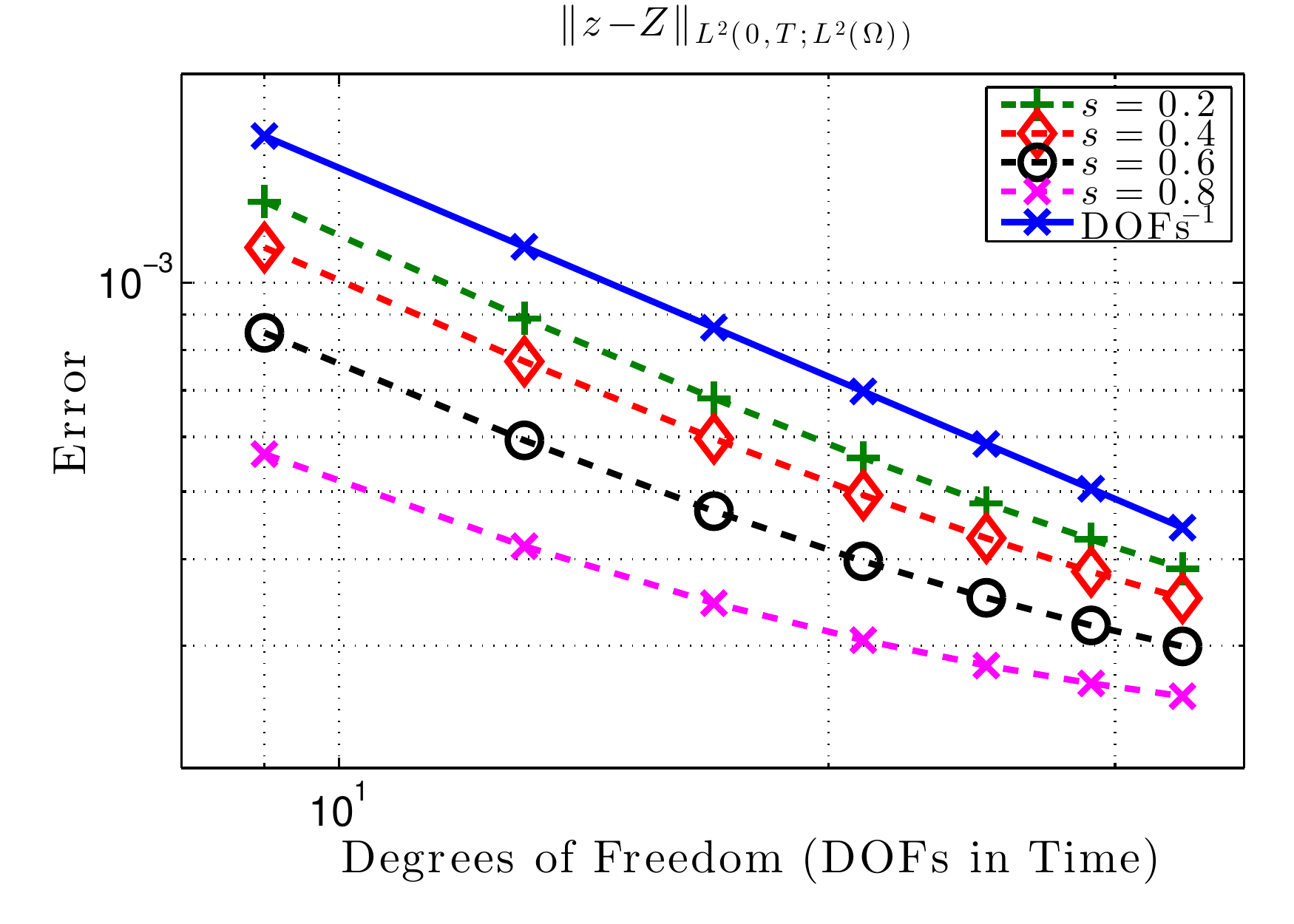} 
\caption{\label{f:time_refine_rate}
Computational rates of convergence for the control on anisotropic meshes for
$n = 2$ and $s = 0.2$, $0.4$, $0.6$ and $s = 0.8$. For a fixed number degrees of freedom in space, $N = 927828$, 
the figure shows the decrease of the $L^2(0,T;L^2(\Omega))$-control error with respect to $\K$. In this case we recover the rate $\K^{-1}$.}
\end{figure}

\bibliographystyle{plain}
\bibliography{biblio}

\end{document}